\documentclass[11pt]{article}

\usepackage{latexsym}
\usepackage{amssymb}
\usepackage{amsthm}
\usepackage{amscd}
\usepackage{tabmac}
\usepackage{amsmath}

\usepackage{mathrsfs}
\usepackage[all]{xy}
\usepackage{hyperref} 
\usepackage[usenames,dvipsnames]{color}
\usepackage{graphicx,eepic}
\usepackage{float}
\usepackage[colorinlistoftodos]{todonotes}

\usepackage{color}

\theoremstyle{definition}
\newtheorem* {theorem*}{Theorem}
\newtheorem{theorem}{Theorem}[section]
\newtheorem{thmdef}[theorem]{Theorem-Definition}

\newtheorem{question}[theorem]{Question}
\theoremstyle{definition}
\newtheorem{observation}[theorem]{Observation}

\newtheorem* {example*}{Example}

\newtheorem{lemma}[theorem]{Lemma}
\theoremstyle{definition}
\newtheorem{definition}[theorem]{Definition}
\theoremstyle{definition}

\newtheorem{conjecture}[theorem]{Conjecture}
\newtheorem{proposition}[theorem]{Proposition}
\newtheorem{corollary}[theorem]{Corollary}

\newtheorem*{remark}{Remark}
\theoremstyle{definition}
\newtheorem {example}[theorem]{Example}
\theoremstyle{definition}

\theoremstyle{definition}

\theoremstyle{definition}

\xyoption{dvips}

\usepackage{fullpage}

\numberwithin{equation}{section}

\def\({\left(}
\def\){\right)}
     \newcommand{\CC}{\mathbb{C}}  \newcommand{\QQ}{\mathbb{Q}}   \newcommand{\cP}{\mathcal{P}} \newcommand{\cA}{\mathcal{A}}
   
\newcommand{\cR}{\mathcal{R}}   \newcommand{\cI}{\mathcal{I}}

\def\NN{\mathbb{N}}
    \def\ZZ{\mathbb{Z}}   
  \def\GL{\mathrm{GL}}       \def\spanning{\textnormal{-span}}

   \newcommand{\supp}{\mathrm{supp}}

\def\fk{\mathfrak}

\def\barr{\begin{array}}
\def\earr{\end{array}}
\def\ba{\begin{aligned}}
\def\ea{\end{aligned}}
\def\be{\begin{equation}}
\def\ee{\end{equation}}

\def\qquand{\qquad\text{and}\qquad}
\def\quand{\quad\text{and}\quad}

\def\I{\mathcal{I}}

\def\DesR{\mathrm{Des}_R}
\def\DesL{\mathrm{Des}_L}

\def\omdef{\overset{\mathrm{def}}}

\def\proj{\mathrm{proj}}

\def\id{\mathrm{id}}
\def\PP{\mathbb{P}}

\def\cG{g}

\def\fkS{\fk S}

\newcommand{\rww}{\operatorname{rrw}}

\def\ben{\begin{enumerate}}
\def\een{\end{enumerate}}

\def\fpf{{\tt {FPF}}}

\def\D{\hat D}
\def\Dfpf{\D_\fpf}

\newcommand{\shift}[2]{1_{#1}\times #2}
\newcommand{\shiftfpf}[3]{\wfpf_{#1} \times #3 \times \wfpf_{#2}}

\def\ellhat{\hat\ell}
\def\Ffpf{\hat F^\fpf}
\def\Sfpf{\hat {\fk S}^\fpf}

\def\Ifpf{\I_\fpf}

\def\cAfpf{\cA_\fpf}
\def\cRfpf{\hat\cR_\fpf}

\def\wfpf{v}

\newcommand{\xRightarrow}[2][]{\ext@arrow 0359\Rightarrowfill@{#1}{#2}}

\newcommand{\Fl}{\operatorname{Fl}}
\newcommand{\Proj}{\operatorname{proj}}

\renewcommand{\O}{\operatorname{O}}
\newcommand{\Sp}{\operatorname{Sp}}
\newcommand{\stb}{\operatorname{stb}}
\newcommand{\Ess}{\operatorname{Ess}}
\newcommand{\rank}{\operatorname{rank}}
\renewcommand{\dim}{\operatorname{dim}}

\renewcommand{\skew}{\operatorname{skew}}
\renewcommand{\r}{\operatorname{rk}}

\def\LambdaOdd{\Gamma}

\makeatletter
\renewcommand{\@makefnmark}{\mbox{\textsuperscript{}}}
\makeatother

\allowdisplaybreaks[1]

\UseCrayolaColors

\begin{document}
\title{Involution words: counting problems and connections to Schubert calculus for symmetric orbit closures}
\author{Zachary Hamaker \footnote{This author was supported by the IMA with funds provided by the National Science Foundation.}
\\
Department of Mathematics \\ University of Michigan \\ { \tt zachary.hamaker@gmail.com}
 \and
   Eric Marberg\footnote{This author was supported through a fellowship from the National Science Foundation.}
\\
 Department of Mathematics \\  HKUST \\ {\tt eric.marberg@gmail.com}
\vspace{3mm}
\and
Brendan Pawlowski\footnote{This author was partially supported by NSF grant 1148634.}
\\ Department of Mathematics  \\ University of Michigan \\ {\tt br.pawlowski@gmail.com}
}

\date{}

\maketitle

\begin{abstract}
Involution words are  variations of reduced words for involutions in Coxeter groups, first studied under the name of ``admissible sequences'' by Richardson and Springer. 
They are maximal chains in Richardson and Springer's weak order on involutions.
This article is the first in a series of papers on involution words, and focuses on their enumerative properties.
We define involution analogues of several objects associated to permutations,
including Rothe diagrams, the essential set,  Schubert polynomials, and Stanley symmetric functions.
These definitions have geometric interpretations for certain intervals in the weak order on involutions.
In particular, our definition of ``involution Schubert polynomials'' can
be viewed as a Billey-Jockusch-Stanley type formula
for cohomology class representatives of $\O_n$- and $\Sp_{2n}$-orbit closures in the flag variety,
defined inductively in recent work of Wyser and Yong.
As a special case of a more general theorem, we show that the involution Stanley symmetric function
for the longest element of a finite symmetric group is a product of staircase-shaped Schur functions.
This implies that the number of involution words for the longest element of a finite
symmetric group is equal to the dimension of a certain irreducible representation of a Weyl group of type $B$.
\end{abstract}

\tableofcontents
\setcounter{tocdepth}{2}

\section{Introduction}

\subsection{Involution words}\label{intro1-sect}

Let $(W,S)$ be a Coxeter system and define $\I = \I(W) = \{ x\in W : x=x^{-1}\}$ to be the set of involutions in $W$. 
A \emph{reduced word} for an element $w \in W$ is a sequence $(s_1,s_2,\dots,s_k)$ with $s_i \in S$ of shortest possible length such that $w=s_1s_2\cdots s_k$. 
An \emph{involution word} for an element $z \in \I$ is a sequence $(s_1,s_2,\dots,s_k)$ with $s_i \in S$ of shortest possible length such that
\be\label{first-eq} 
z=(\cdots ((1 \rtimes s_1) \rtimes s_2) \rtimes \cdots )\rtimes s_k
\ee
where for $g \in W$ and $s \in S$ we let $g \rtimes s $ be either $gs$ (if $s$ and $g$ commute) or $sgs$ (if $sg \neq gs$). %
When $g \in \I$, the element $g\rtimes s$ is also an involution.
Less obviously, every $z \in \I$ has at least one involution word with the convention that the empty sequence $\emptyset$ is the unique involution word of 
the identity element $1 \in \I$.
We write $\cR(w)$ for the set of reduced words of  $w \in W$ and  $\hat \cR(z)$ for the set of involution words of  $z \in \I$.
Given any involutions $y,z \in \I$, we define $\hat\cR(y,z)$ to be the set of sequences in $S$ which, when appended to   involution words for $y$,  produce involution words for $z$. The set $\hat\cR(y,z)$ may be empty, and we refer to its elements as \emph{involution words} from $y$ to $z$.

Involution words have many properties analogous to those of  ordinary reduced words, which accounts for our terminology. 
Reduced words correspond to maximal chains in $W$ under the right weak order.
Involution words analogously correspond to maximal chains in $\cI$ under the \emph{involution weak order}  defined by Richardson and Springer in \cite[\S3.17]{R}.
For initial intervals (that is, intervals starting at $y=1$), involution words are the same as what Hultman calls ``reduced $\underline S$-expressions'' in \cite{H2,H3} and are the right-handed versions of   ``admissible sequences'' in \cite{R,RS} and ``reduced $\textbf{I}_*$-expressions'' (with $*=\id$) in \cite{EM1,EM2}.

For permutations, the involution weak order can be identified with the weak order on the set of $B$-orbit closures in certain spherical varieties, and involution words are studied in this form  by Can, Joyce, and Wyser in \cite{CJ,CJW,CJW2}.
Specifically, the orbits induced by the actions of the orthogonal and symplectic groups on the flag variety have weak orders whose saturated chains starting at the largest orbit correspond to involution words in the intervals starting at $1$ and $v_n = [2,1,4,3,\dots,2n,2n{-}1] \in S_{2n}$, respectively.
These  \emph{geometric cases} are of particular interest, and lead us to define, alongside $\hat\cR(y)$, the  set 
\be\label{Rfpf-def}
\cRfpf(z) \omdef= \hat \cR(v_n,z) \qquad \text{for $z \in \I(S_{2n})$}.
\ee
Elements of this set will be called \emph{fixed-point-free involution words}, since $v_n$ is the involution of smallest possible length with no fixed points in $S_{2n}$.
The set $\cRfpf(z)$ is non-empty if and only if $z \in \I(S_{2n})$ has no fixed points, 
in which case every involution in the interval between $v_n$ and $z$ in weak order will also be fixed-point-free.
Fixed-point-free involution words are a special case of Rains and Vazirani's
notion of ``reduced expressions'' for elements of \emph{quasiparabolic sets} \cite{RV}.

\subsection{Geometry of $K$-orbits in flag varieties}\label{intro2-sect}

Before describing our results on involution words, we provide a brief overview of the geometry underlying the   geometric cases.
Let $B$ be the Borel subgroup of lower triangular matrices in $\GL_n(\CC)$ and denote by $\Fl(n) = B\backslash \GL_n(\CC)$ the flag variety. 
The right  orbits of the opposite Borel subgroup $B^+$ in $\Fl(n)$ decompose into Schubert cells, whose Zariski closures are the \emph{Schubert varieties} $X_w$, which are indexed by permutations $w \in S_n$.
Schubert varieties can also be defined explicitly using a fixed reference flag and rank conditions determined by $w$. 
By instead considering the right actions on $\Fl(n)$ of another group $K$, such as the  {orthogonal group} $\O_n(\CC)$ or (when $n$ is even) the  {symplectic group} $\Sp_n(\CC)$, one obtains different orbit decompositions.
The $K$-orbits in $\Fl(n)$ are certain sets $Y^K_y$ indexed by arbitrary involutions $y \in \I(S_n)$ when $K=\O_n(\CC)$ and by fixed-point-free involutions in $S_n$ when $K = \Sp_n(\CC)$.
Again, $Y^K_y$ can be defined using a fixed reference flag and explicit rank conditions determined by $y$.

Each Schubert variety determines a class $[X_w]$ in the cohomology ring $H^*(\Fl(n),\ZZ)$, which can be identified with the quotient of $\ZZ[x_1,\dots,x_n]$ by the ideal generated by the symmetric polynomials of positive degree via the Borel isomorphism (see \eqref{borel-eq}).
The \emph{Schubert polynomial} $\fkS_w$, defined by Lascoux and Sch\"utzenberger~\cite{LS1982a}, is a particularly nice choice of representative under this map for the cohomology class of the Schubert variety $X_w$.
When $K = \O_n(\CC)$ or $\Sp_n(\CC)$, we may consider the cohomology class of $Y^K_y$ in $H^*(\Fl(n),\ZZ)$  as an element of the same quotient of $\ZZ[x_1,\dots,x_n]$.
For these  classes $[Y^K_y]$, Wyser and Yong have defined similarly nice polynomial representatives $\Upsilon^K_{y}$, which they call  $\Upsilon$-polynomials~\cite{WY}.
The construction of $\Upsilon^K_y$ in \cite{WY} relies on first choosing a representative for the class of the unique closed orbit, and then showing compatibility with certain compositions of divided difference operators.

\subsection{Atoms for involutions}

This paper initiates the study of involution words from an enumerative perspective.
We introduce ``involution'' analogues of Rothe diagrams and Fulton's essential set for the geometric cases, and of Schubert polynomials and Stanley symmetric functions.
Some of these definitions are simple extensions of the ordinary versions in light of the following fundamental result.

\begin{thmdef}
[{\cite[Lemma 3.16]{R}}]
\label{atoms-thmdef}
For each  $y,z \in \I(W)$,  there exists a finite subset 
$
\cA(y,z)\subset W
$
such that $\hat \cR(y,z) = \bigcup_{w \in \cA(y,z)} \cR(w).$
Equivalently, every involution word from $y$ to $z$ is a reduced word for some element of $W$ and the set $\hat \cR(y,z)$ is closed under the braid relations for $(W,S)$.
For $y,z \in \I$ and $w \in \cA(y,z)$, we say that $w$ is a \emph{relative atom} from $y$ to $z$.
\end{thmdef}

In the geometric cases, we define $\cA(y) \omdef= \cA(1,y)$ and $\cAfpf(z) \omdef= \cA(v_n,z)$.

\begin{remark}
The theorem-definition follows  from  results of Richardson and Springer \cite{R}. A direct proof using our present notation appears in \cite{HMP2}.
\end{remark}

The \emph{involution Rothe diagrams} $\D(y)$ and $\Dfpf(y)$ of $y \in \I(S_n)$ are defined in Section~\ref{invdiagram-sect} as certain restrictions of the usual Rothe diagram $D(y)$.
The \emph{essential sets} $\Ess(\D(y))$ and $\Ess(\Dfpf(y))$ consist of southeast corners in the corresponding involution diagram.
This closely mirrors the definition of Fulton's essential set $\Ess(D(w))$ for $w \in S_n$.
In Proposition~\ref{prop:essential-set}, we show that the involution essential sets determine a subset of the rank conditions sufficient to define $Y^K_y$ when $K = \O_n(\CC)$ or $K=\Sp_n(\CC)$.
The proof is largely a consequence of the analogous result for the $B^+$-action, with some  subtleties in the fixed-point-free case.
These objects prove to be a key tool in our study of involution Schubert polynomials and involution Stanley symmetric functions.

\subsection{Schubert polynomials and Stanley symmetric functions}

Schubert polynomials were originally defined using divided difference operators.
However, they can also be viewed as a sort of generating function over reduced words.
More specifically, Billey, Jockusch and Stanley~\cite{BJS} and Fomin and Stanley~\cite{FS} found the following explicit combinatorial formula.
Let $s_i$ denote the simple transposition $(i,i+1)$, so that $S_n$ is a Coxeter group relative to the generating set  $\{s_1,s_2,\dots,s_{n-1}\}$.
Fix $w \in S_n$, and for each $a = (s_{a_1},s_{a_2},\dots,s_{a_k}) \in \cR(w)$, let $C(a)$ be the set of sequences  of positive integers
 $I = (i_1,i_2\dots,i_k) $ satisfying 
 \[
 i_1 \leq i_2 \leq \dots \leq i_k \qquand i_j < i_{j+1}\text{ whenever }a_j < a_{j+1}.\] 
 We write $I \leq a$ to indicate that  $i_j \leq a_j$ for all $j $ and define $x_I = x_{i_1} x_{i_2}\cdots x_{i_k}$.
The Schubert polynomial corresponding to $w \in S_n$ is then given by
\be\label{schub1-eq}
 \fkS_w \omdef= \sum_{a \in \cR(w)} \sum_{ \substack{ I \in C(a) \\ I \leq a} } x_I \in \ZZ[x_1,\dots,x_n].
 \ee
This formula makes clear that $\fkS_w$ is homogeneous with degree equal to the length of $w$.
Similarly, the \emph{Stanley symmetric function} of $w$ is
 \be
 \label{F1-eq}
 F_w \omdef= \sum_{a \in \cR(w)} \sum_{I \in C(a)} x_I  \in \ZZ[[x_1,x_2,\dots]]
 \ee
(this definition is $F_{w^{-1}}$ in~\cite{Stan}).
The coefficient  of $x_1 x_2 \dots x_{\ell(w)}$ in $F_w$ is $|\cR(w)|$,
and it holds that $F_w =\lim_{N\to\infty} \fkS_{1_N \times w}$
where $1_N \times w$ is the image of $w$ under the natural embedding $S_n \hookrightarrow S_N\times S_n \subset S_{N+n}$ and the limit is taken in the sense of formal power series.
This limit is called \emph{stabilization}, and Stanley symmetric functions are sometimes referred to as \emph{stable Schubert polynomials}.

\subsection{Main results}\label{intro5-sect}

For $y,z \in \I(S_n)$,  we define the \emph{involution Schubert polynomial} $\hat \fkS_{y,z} $ and \emph{involution Stanley symmetric function} $\hat F_{y,z}$ 
by the formulas
\be
\label{intro-inv-schub-eq}
\hat \fkS_{y,z} \omdef= \sum_{a \in \hat \cR(y,z)} \sum_{ \substack{ I \in C(a) \\ I \leq a} } x_I
\qquand 
\hat F_{y,z} \omdef= \lim_{N\to\infty} \hat \fkS_{1_N \times y,1_N \times z} = \sum_{a \in \hat \cR(y,z)} \sum_{I \in C(a)} x_I.
\ee
Theorem-Definition~\ref{atoms-thmdef} implies that \[\hat \fkS_{y,z} = \sum_{w \in \cA(y,z)} \fkS_w \qquand \hat F_{y,z} = \sum_{w \in \cA(y,z)} F_w.\]
For the geometric cases, we define 
\[\hat \fkS_y \omdef= \hat \fkS_{1,y},\qquad \hat F_y \omdef= \hat F_{1,y},\qquad \Sfpf_y \omdef= \hat \fkS_{v_n,y},\qquand \Ffpf_y\omdef= \hat F_{v_n,y}.\]
As one would hope, these involution Schubert polynomials are the same (up to scaling factor) as Wyser and Yong's representatives for $[Y^K_y]$.
Let $\kappa(y)$ be the number of two-cycles in $y \in \I(S_n)$.

\begin{theorem}
\label{t:WyserYong}
For each $y \in \I(S_n)$ and each fixed-point-free $z \in \I(S_{2n})$,
it holds that
\[
2^{\kappa(y)}\hat \fkS_y = \Upsilon^{\O_n}_y \qquand  \Sfpf_z = \Upsilon^{\Sp_{2n}}_z.
\]
\end{theorem}

Our proof of  this result, which is restated as Theorem~\ref{WYatom-thm},
proceeds by generalizing a characterization of Schubert polynomials to the involution setting.
This allows us to show that our formula coincides with   Wyser and Yong's formula for the class of the unique closed orbit, and that it behaves in the same way with respect to divided difference operators.

A general formula of Brion~\cite[Theorem 1.5]{Brion98} shows that (up to a power of $2$) $[Y^K_y]$ is a multiplicity-free sum of Schubert classes $\sum_{w \in \cA^K(y)} [X_w]$. 
Wyser and Yong \cite[Section 5]{WY} note that their representative for $[Y^K_y]$ is a linear combination of Schubert polynomials, and therefore is equal to $\sum_{w \in \cA^K(y)} \fkS_w$ (again up to a power of $2$). 
As a consequence, one gets an analogue of \eqref{schub1-eq} for $\Upsilon^K_y$ by replacing $\cR(w)$ with 
$\bigcup_{w\in \cA^K(y)} \cR(w)$.
From this point of view, the main contribution of Theorem~\ref{t:WyserYong} is combinatorial: we identify $\bigcup_{w \in \cA^K(y)} \cR(w)$
with either $\hat \cR(y)$ or $\cRfpf(y)$ and $\cA^K(y)$ with either $\cA(y)$ or $\cAfpf(y)$.

In both geometric cases, the longest permutation $w_n = [n,n{-}1,\dots,1] \in S_n$ indexes the orbit of the fixed reference flag.
For this closed orbit, Wyser and Yong's polynomial representatives are
\be \label{eq:wyser.yong.prod}
\Upsilon^{\O_n}_{w_n} = 2^{\lfloor \frac{n}{2} \rfloor} \hat \fkS_{w_n} = \prod_{1 \leq i \leq j \leq n{-}i} (x_i + x_j) \ \ \ \mbox{and} \ \ \ \Upsilon^{\Sp_{2n}}_{w_{2n}} = \Sfpf_{w_{2n}} = \prod_{1 \leq i < j \leq 2n{-}i} (x_i + x_j).
\ee
A permutation is \emph{dominant} if it is 132-avoiding.
Another of our main results is to extend the product formulas \eqref{eq:wyser.yong.prod} to  {dominant}  involutions as follows.
\begin{theorem}
\label{t:dominvSchub}
Let $y \in \I(S_n)$  and let $z \in \I(S_{2n})$ be fixed-point-free. If $y$ and $z$ are dominant, 
then
\[
\hat \fkS_y = 2^{-\kappa(y)}\prod_{(i,j) \in \D(y)} (x_i + x_j) \qquand \Sfpf_z = \prod_{(i,j) \in \Dfpf(z)} (x_i + x_j)
\]
where $\D(y)$ and $\Dfpf(z)$ are defined as in Section \ref{invdiagram-sect}.
\end{theorem}

Theorem~\ref{t:dominvSchub} is restated as Theorem~\ref{invSchuprod-thm} and is a special case of Theorem~\ref{factor-thm},
which describes a product formula for the involution Schubert polynomials of a more general class of permutations that we call \emph{weakly dominant}.

In~\cite{Stan}, Stanley showed that power series $F_w$ are symmetric
and computed several of these functions explicitly.
He was able to show, for example, that $F_{w_n}$ is the  Schur function $s_{\delta_n}$ indexed by the \emph{staircase shape partition}
$\delta_n=(n-1,n-2, \dots, 1)$.
This implies  that $|\cR(w_n)|$ is equal to  $f^{\delta_n},$ the number of \emph{standard Young tableaux} of shape $\delta_n$.
More generally, as a consequence of work by Lascoux and Sch\"utzenberger~\cite{LS1982a},
and as proven bijectively by Edelman and Greene~\cite{EG}, Stanley symmetric functions are \emph{Schur positive},
i.e., they can be expressed as positive integer sums of Schur functions.
Since involution Stanley symmetric functions are sums of Stanley symmetric functions, they inherit this property.

In the geometric cases, we characterize the involutions whose involution Stanley symmetric function is a single Schur function.
Moreover, by carefully studying the stabilization of certain weakly dominant involution Schubert polynomials, we obtain expressions for the corresponding involution Stanley symmetric functions.
Most notably we derive the following result,  which was conjectured in 2006 in unpublished work of Cooley and Williams \cite{Wcorr}.
\begin{theorem} \label{F-thm}
Let $ p = \lceil \frac{n+1}{2}\rceil$ and $q= \lfloor \frac{n+1}{2}\rfloor$, and set $P = \binom{p}{2}$, $Q = \binom{q}{2}$, and $N = \binom{n}{2}$.
Then
\[
\hat F_{w_n} = s_{\delta_p} s_{\delta_q}
\qquand
\Ffpf_{w_{2n}} = (s_{\delta_{n}})^2.
\]
Consequently,
$
|\hat \cR(w_n)| = \tbinom{P+Q}{P} f^{\delta_p} f^{\delta_q}
$
and
$
|\cRfpf(w_{2n})| = \tbinom{2N}{N}(f^{\delta_{n}})^2.$
\end{theorem}
Theorem~\ref{F-thm} is a special case of Theorem~\ref{last-thm}, which provides product formulas for a certain family of weakly dominant involutions.
Every involution Stanley symmetric function computed in Theorem~\ref{last-thm} is \emph{Schur-$P$ positive}.
In later work~\cite{HMP4,HMP5}, we present proofs that $\hat F_y = \hat F_{1,y}$ and $\Ffpf_z$ are Schur-$P$ positive for all $y \in \I(S_n)$ and $z \in \Ifpf(S_{2n})$.
We do not yet have a good understanding of when the symmetric function $\hat F_{y,z}$ is Schur-$P$ positive.
It can happen that an involution Stanley symmetric function is not expressible using Schur-$P$ functions.
For example, $\hat F_{[2,1,3,4],[3,4,1,2]} = s_{(1,1)}$ is not in the ring generated by Schur-$P$ functions. 
\begin{question}
For which $y,z \in \I(S_n)$ is $F_{y,z}$ Schur-$P$ positive?

\end{question}

Although our enumerative results are restricted to the geometric cases for the symmetric group, the objects we study have natural analogues in other Coxeter groups.
Several questions remain in this direction.
For example, Haiman showed in  \cite{Haiman} that $|\cR(w^B_n)| = f^{(n^n)}$ where $w^B_n$ is the longest element in the Weyl group $B_n$ and $(n^n)=(n,n,\dots,n)$.
Computations suggest  the following version of this theorem for involution words.

\begin{conjecture}
The set $\hat\cR(w_n^B)$ has size $f^{\delta_{n+1}}$.
\end{conjecture} 

\begin{remark}[Note added in proof]
A proof of this conjecture will appear in \cite{MP}. 
\end{remark}

Additionally, there is a notion of \emph{twisted involution words} for which Schubert polynomial and Stanley symmetric function analogues are readily defined.
We do not explore these objects in the present paper, but it remains a question of interest to find geometric interpretations for twisted involution Schubert polynomials.

\subsection*{Outline}

The rest of the paper is structured as follows:
\begin{itemize}
\item
Section~\ref{prelim-sect} reviews some general properties of involution words, Rothe diagrams, Schubert polynomials, and Stanley symmetric functions.

\item
Section~\ref{results-sect} contains our main results.
In Section~\ref{Symatoms-sect}, we describe some noteworthy facts about atoms for permutations. 
Section~\ref{invdiagram-sect}
introduces analogues of Rothe diagrams and
codes for involutions.
Sections~\ref{invschubert-sect}, \ref{revisit-sect}, and \ref{product-sect} primarily concern
involution Schubert polynomials.
Sections~\ref{invstan-sect} and \ref{stabilization-sect} are mostly about involution Stanley symmetric functions.

\item
Finally, Appendix~\ref{not-sect} provides an index of notation.
\end{itemize}

\subsection*{Acknowledgements}

We are especially grateful to
Dan Bump and Vic Reiner for many helpful conversations in the course of the development of this paper.
We also thank
 Sara Billey, Michael Joyce, Joel Lewis, J{\o}rn Olsson, Ben Wyser, Alex Yong, and the anonymous referees for  useful discussions and suggestions.

\section{Preliminaries}\label{prelim-sect}

Write $\PP = \{1,2,3,\dots\}$  for the positive integers and define $\NN = \{0\}\cup \PP$ and  $ [n] = \{ i \in \PP : i\leq n\}$.
If $(W,S)$  is a Coxeter system, then we write $\ell : W \to \NN$ for its length function, and 
denote by
\be\label{des-eq} \DesL(w) \omdef= \{ s \in S : \ell(sw) < \ell(w)\} \qquand \DesR(w) \omdef= \{ s \in S: \ell(ws) < \ell(w)\}\ee
 the  \emph{left} and \emph{right descent sets} of an element $w \in W$.

\subsection{General properties of involution words}\label{gen-sect}

Here we review the  basic properties of   involution words for an arbitrary Coxeter group. Most of this material   appears in some form  in \cite{R,RS,Springer} or  the more recent papers \cite{H1,H2,H3}.

\begin{remark}
Our definition of involution words has a straightforward generalization to \emph{twisted involutions} in Coxeter groups, by which we mean  elements $w \in W$ satisfying $w^{-1} = w^*$ for some fixed $S$-preserving automorphism $*$ of $W$ of order two. This more flexible setup is the point of view of our references, but our present applications   will not require this generality. 
\end{remark}

Let  $(W,S)$  be any Coxeter system
and write 
$\I = \I(W) \omdef= \{ w \in W : w^{-1} = w\}$.
Recall from the start of the introduction that we define
\be\label{rtimes-eq} y\rtimes s \omdef = \begin{cases} sys &\text{if $ys \neq sy$} \\ ys&\text{otherwise}\end{cases}
\qquad\text{for $y \in \I$ and $s \in S$}.
\ee

\begin{remark}
 Although $(y\rtimes s)\rtimes s = y$ for  $s \in S$, 
the operation $\rtimes : \I \times S \to \I$ does not  extend to a right $W$-action in general:
 if $s,t \in S$ and $sts=tst$ then  $((1\rtimes s) \rtimes t) \rtimes s = t $ but $ ((1\rtimes  t)\rtimes s) \rtimes t = s$.
We   omit all parentheses in  expressions like \eqref{first-eq} and interpret 
$1 \rtimes s_1 \rtimes s_2 \rtimes \cdots \rtimes s_k$ to mean $(\cdots ((1 \rtimes s_1) \rtimes s_2) \rtimes \cdots )\rtimes s_k$,
which is the only sensible way of parenthesizing the former expression.
\end{remark}

Define  $\cR(w)$, $\hat\cR(z)$, and $\hat\cR(y,z)$ for $w \in W$ and $y, z\in \I$ as in the introduction.
For $y,z \in \I$, the set $\hat\cR(y,z)$ consists of all words $(s_1,\dots,s_k)$ with $s_i \in S$ such that
for some (equivalently, every) word $(r_1,\dots,r_j) \in \hat\cR(y)$ it holds that
 $(r_1,\dots,r_j,s_1,\dots,s_k) \in \hat\cR(z).$
 We have $\hat\cR(y,y) = \{ \emptyset\}$ where $\emptyset$ denotes the empty sequence.
 The set $\hat\cR(y,z)$ may be empty, for example if $\ell(y) > \ell(z)$.
 
Fix $y \in \I$ and $s \in S$.
It is a  consequence of the exchange principle that   $\ell(sys) = \ell(y)$ if and only if $sys=y$ \cite[Lemma 3.4]{H2},
and so if  $s \in \DesR(y)$ then
\be\label{ides-eq}
\ell(y\rtimes s) = \begin{cases} \ell(y) -2 &\text{if $y\rtimes s = sys$}
\\
\ell(y)-1 &\text{if $y \rtimes s = ys$}.
\end{cases}
\ee
From this property, it follows  by induction on  length that  $\hat\cR(y) \neq \varnothing$ for all $y \in \I$, so we may
set 
\be\label{ellhat-def}
\ellhat(y)\omdef=\text{the common length of all involution words for $y \in \I$.}
\ee
We also define $\ellhat(y,z) \omdef = \ellhat(z)- \ellhat (y)$. If the set $\hat\cR(y,z)$ is nonempty,
then $\ellhat(y,z)$ is the common length of all of its elements.

\begin{remark}\label{kappa-rmk}
The map $\ellhat : W\to \NN$ is denoted $L$ in \cite[\S3]{R} and $\rho$ in \cite{H1,H2,H3}. 
Incitti \cite{Incitti1,Incitti2} has derived useful combinatorial formulas for  $\ellhat $ when $W$ is a classical Weyl group. In the case when $W=S_n$ is a symmetric group, one has
\[ \ellhat(y) = \tfrac{1}{2} \( \ell(y)+ \kappa(y)\)\qquad\text{for $y \in \I(S_n)$}\]
where $\ell(y)$ is the usual length and $\kappa(y)$ is the number of 2-cycles of the involution $y$.
\end{remark}

The  \emph{(strong) Bruhat order} of $(W,S)$ is the partial order $\leq$ on $W$ in which $u\leq v$ if and only if in each reduced expression for $v$ one can omit a certain number of factors to obtain a reduced expression for $u$.
Thus  $u< v$  implies $\ell(u) < \ell(v)$, and  it follows from \eqref{ides-eq} that if $y \in \I$ and $s \in \DesR(y)$ then   $y\rtimes s \leq y s < y$.
There is a close relationship between the  the Bruhat order on $\I$ and involution words.
For example,   $(\I,\leq)$ is a graded poset with rank function $\ellhat : \I \to \NN$
 \cite[Theorem 4.8]{H1},
and this poset
 inherits  the subword characterization of $(W,\leq)$ given above, but with the role of reduced words  replaced by involution words \cite[Theorem 2.8]{H3}.
From these results, it is clear that 
if $y \in \I$ and $s \in S$ then 
the following are equivalent:
\[  y \rtimes s < y\quad  \Leftrightarrow\quad \ell(ys) = \ell(y)-1 \quad \Leftrightarrow\quad \ell(y\rtimes s) < \ell(y)\quad \Leftrightarrow\quad \ellhat(y\rtimes s) = \ellhat(y)-1.\]
These properties imply the following useful alternative definition of the set  $\hat\cR(y,z)$:

\begin{lemma}\label{increasing-cor} If $y,z \in \I$, then a word $(s_1,s_2,\dots,s_k)$ with $s_i \in S$ belongs to $\hat\cR(y,z)$  if and only if 
\[y  < y_1 < y_2< \dots <y_k = z\qquad\text{where }y_i = y \rtimes s_1\rtimes s_2 \rtimes \cdots \rtimes s_i.\]
\end{lemma}

Recall that $\cRfpf(y) = \hat \cR(\wfpf_n,y)$ where $\wfpf_n = s_1s_3\cdots s_{2n-1}$.
\begin{corollary}\label{rfpf-cor}
If $y \in \I(S_{2n})$ then 
$ \cRfpf(y)
 $
is non-empty if and only if $y$ is fixed-point-free.
\end{corollary}

\begin{proof}
If $y \in S_{2n}$ is a fixed-point-free involution and $s$ is a simple transposition then $sys$ is also fixed-point-free while  $y \rtimes s = y s$ only if $s \in \DesR(w)$; then invoke the preceding lemma.
\end{proof}

The set of relative atoms $\cA(y,z)$ for $y,z \in \I(W)$ is defined in Theorem-Definition~\ref{atoms-thmdef}.
The properties of $\cA(y,z)$ are the focus of our paper \cite{HMP2}. 
While $\cA(y,z)$  may be empty, the set $\cA(y) = \cA(1,y)$ is always nonempty, and we have $\cA(y,y)=\cA(1)= \{1\}$.
It is clear that  
\[
\cA(y,z) = \{ w \in W : \ell(w) =\ellhat(y,z)\text{ and }vw \in \cA(z)\text{ for some }v \in \cA(y)\},
\] 
so $\cA(y,z)$ can be computed from $\cA(y)$ and $\cA(z)$.

\begin{example}\label{atom-ex}
For the involutions $\wfpf_2 = [2,1,4,3]$ and $w_4 = [4,3,2,1]$ in $S_4$ we have 
\[\cA(w_4) = \{ [2, 4,3,1], [3,4,1,2], [4,2,1,3] \}
\qquand
\cA(\wfpf_2,w_4) = \{ [1,3,4,2], [3,1,2,4]\}
.\] In general, the sets $\cA(w_n)$ and $\cA(\wfpf_n,w_{2n})$ have cardinality $(n-1)!!$ and $ n!$ and are given by a simple recursive construction due to Can, Joyce, and Wyser \cite{CJ,CJW}.
\end{example}

\begin{proposition}[{\cite[Proposition 2.8]{HMP2}}] \label{atomdes-prop}
 Let $y,z\in \I$ and $s \in S$.
 \ben
 \item[(a)] If $s \notin \DesR(z)$ then $\cA(y,z  ) = \{ ws : w \in \cA(y,z\rtimes s) \text{ and } s \in \DesR(w) \}.$
  \item[(b)] If $s \in \DesR(y)$ then $\cA(y ,z) = \{ sw : w \in \cA(y\rtimes s,z) \text{ and } s \in \DesL(w) \}.$
  \een
  Consequently, if $u \in \cA(y,z)$ then
 $\DesR(u) \subset \DesR(z) $ and $ \DesL(u) \subset S\setminus\DesR(y).$
\end{proposition}

We mention another order on $\I$ which will be of relevance.
The \emph{left} and \emph{right weak orders} $\leq_L$ and $\leq_R$ on $W$ are the transitive closures of the relations $w<_L sw$ and $w <_R wt$ for  $w \in W$ and $s,t \in S$ such that  
$\ell(sw) > \ell(w)$ and $\ell(wt) > \ell(w)$.
Following \cite[Section 5]{H2}, we define the \emph{(two-sided) weak order} $\leq_T$ on $\I$  to be the transitive closure of the relations 
\be\label{t-eq} w <_T w \rtimes s\qquad\text{for $w \in \I$ and $s \in S$ such that $\ellhat(w)  < \ellhat(w\rtimes s) $.}\ee
Evidently $\cA(y,z)$ is nonempty if and only if $y\leq_T z$, and each element of $\cR(y,z)$ corresponds  to a maximal chain from $y$ to $z$ in the poset $(\I,\leq_T)$.
If $y,z \in \I$ then 
$ y \leq_T z $ implies $ y \leq z$,
but the  reverse implication does not hold in general.

\subsection{Diagrams and codes for permutations}
\label{diagram-sect}

We write $S_\infty$ for the group of permutations $w$ of $\PP$ whose \emph{support} $\supp(w) = \{ i \in \PP : w(i) \neq i\}$ is finite, and  identify $S_n$ for $n \in \PP$ with the subgroup of permutations $w \in S_\infty$ with $\supp(w)\subset [n]$.
The group $S_\infty$ is a Coxeter group with respect to the generating set $\{s _i = (i,i+1) : i \in \PP\}$.
The right descent set \eqref{des-eq} of $w \in S_\infty$ is
\be\label{Des-def}
\DesR(w) = \{ s_i  : i\in\PP\text{ and } w(i)>w(i+1)\}.
\ee
We say that $i$ is a \emph{descent} of $w \in S_\infty$ if $w(i)>w(i+1)$, so that $s_i \in \DesR(w)$.

The 
 \emph{Rothe diagram}  (see \cite[\S2.1.1]{Manivel})
 of  $w \in S_\infty$ is the set
\be\label{rothe-eq} D(w) \omdef=\left\{ (i,j)  \in \PP \times \PP : j < w(i)\text{ and }i < w^{-1}(j) \right\}.\ee
The set $D(w)$ is obtained by applying the map $(i,j) \mapsto (i,w(j))$ to the   \emph{inversion set} of $w$ given by
\[ \mathrm{Inv}(w) \omdef= \left\{ (i,j) \in \PP\times \PP : i <j \text{ and }w(i) > w(j)\right\}.\]
Consequently $D(w^{-1}) = D(w)^T$ where $T$ denotes the transpose map $(i,j) \mapsto (j,i)$.
If $w \in S_n$ has largest descent $k$, then $D(w) \subset [k]\times [n]$.

\begin{example}\label{Rothe-ex}
We have
$
D(\wfpf_n) = \{ (2i-1,2i-1) : i \in [n]\}$
and
$
D(w_n) = \{ (i,j) \in \PP^2 : i+j \leq n\}$.
\end{example}

The \emph{diagram} of an integer partition $\lambda = (\lambda_1\geq \lambda_2 \geq \dots)$ is the set $\{ (i,j) \in \PP \times \PP : j \leq \lambda_i \}.$
We often identify partitions with their diagrams, and write $(i,j) \in \lambda$ to indicate that $(i,j)$ belongs to the diagram of $\lambda$. If $\lambda$ and $\mu$ are   partitions with $\mu \subset \lambda$
 then the \emph{skew shape} $\lambda/\mu$ is  the complement of the diagram of $\mu$ in the diagram of $\lambda$. The \emph{shifted shape} of a strict partition $\lambda$ (i.e., a partition with distinct parts) is the set 
$\{ (i, j+i-1) : (i,j) \in \lambda\}.$
Two  finite subsets of $\PP\times\PP$ (in particular,   Rothe diagrams or diagrams of partitions or skew shapes or shifted shapes) are \emph{equivalent} if one can be transformed to the other by permuting its rows and then its columns.

\begin{example}
If $\lambda = (2,2,1)$ and $\mu = (1)$ then $\{(1,1),(1,3),(2,3),(3,1)\}$ is equivalent to $\lambda/\mu$.
\end{example}

The \emph{code} of  $w \in S_n$ is the sequence $c(w) \omdef= (c_1(w),c_2(w),\dots,c_n(w)) \in \NN^n$ where 
\be\label{code-def} c_i(w) = | \{ j \in [n]: i<j \text{ and }w(i)>w(j)\}|.\ee
Observe 
that $c_i(w)$ is the number of cells in the $i^{\mathrm{th}}$ row of $D(w)$.
The \emph{shape} $\lambda(w)$ of $w \in S_n$ is the    partition of $\ell(w)$ whose parts are the  nonzero entries of $c(w)$. 

\begin{example}
If $w =[3,7,4,1,6,5,2]$ then
$   c(w)= (2,5,2,0,2,1,0) $ and $ \lambda(w) = (5,2,2,2,1)$,
\end{example}

\subsection{Schubert polynomials}\label{schubert-sect}

We sketch here the  fundamental properties of the \emph{Schubert polynomials} $\fkS_w$ as defined in the introduction; our main references are \cite{Knutson,Manivel}.
We write
\be\label{cP-eq}
\cP_n \omdef= \ZZ[x_1,x_2,\dots,x_n]
\qquand
\cP_\infty \omdef=   \ZZ[x_1,x_2,\dots]
\ee
 for the rings of polynomials in finite and countable sets of commuting variables $\{x_1,x_2,\dots\}$.
The group $S_n$ (respectively, $S_\infty$) acts on $\cP_n$ (respectively $\cP_\infty$) by permuting variables. 
With respect to this action, the \emph{divided difference operator} $\partial_i $  for $i \in \PP$ is defined by
\be\label{partial-i-eq}\partial_i f \omdef= (f-s_i f)/(x_i-x_{i+1}) \qquad\text{for }f \in \cP_\infty.\ee
For example, $\partial_i (x_i^3) = x_i^2 +x_ix_{i+1} + x_{i+1}^2$.
It is a standard exercise to check that
this formula in fact gives a linear map $\partial_i :\cP_\infty \to \cP_\infty$,
and that 
\be
\label{+eq}
\partial_i(fg) = f \cdot \partial_i g\qquad\text{if $f,g \in \cP_\infty$ and $s_i f= f$.}
\ee
One may characterize the Schubert polynomials without explicitly constructing them.
\begin{theorem}[{\cite[Theorem 2.3]{Knutson}}] \label{schubunique-thmdef}
The Schubert polynomials $\{ \fkS_{w} \}_{w \in S_\infty}$  are the unique family of homogeneous polynomials indexed by the elements of $S_\infty$ such that 
\[\fkS_{1} = 1
\qquand
\partial_i \fkS_{w} = \begin{cases} \fkS_{w s_i} &\text{if $s_i\in \DesR(w)$}\\ 0 &
\text{otherwise}\end{cases}
\quad\text{for all } i \in \PP.
\]
\end{theorem}

  The divided difference operators satisfy $\partial_i^2=0$ as well as the Coxeter relations for $S_\infty$ given by
\be\label{coxrel}
\partial_i \partial_{i+1} \partial_i = \partial_{i+1} \partial_i \partial_{i+1} \qquand \partial _i \partial_j = \partial_j\partial_i
\qquad\text{for $i,j \in \PP$ with $|i-j| > 1$.}
\ee
 For $w \in S_\infty$, we may thus define $\partial_w \omdef = \partial_{i_1} \partial_{i_2}\cdots \partial_{i_k}$ for any reduced word $(s_{i_1},s_{i_2},\dots,s_{i_k}) \in \cR(w)$.

  \begin{theorem}[{See \cite[\S2.3.1]{Manivel}}]\label{fkSw0-thm}
  If
   $n \in \PP$ and $v \in S_n$  and $w_n=[n,n-1,\dots,3,2,1] \in S_n$, then 
\[\fkS_{w_n} = x_1^{n-1}  x_2^{n-2} x_3^{n-3}\cdots x_{n-1}
\qquand
  \fkS_v = \partial_{v^{-1} w_n} \fkS_{w_n}.
  \]
\end{theorem}

Let $y = \{y_1,y_2,\dots\}$ be another countable set of commuting variables, which commute also with $x=\{x_1,x_2,\dots\}$.
 If $f \in \cP_\infty$ then we write $f(y)$ to denote the polynomial given by evaluating $f$ at $x_i=y_i$, and for emphasis we sometimes write $f=f(x)$.\
We let $\cP_\infty(x;y) = \ZZ[x_1,y_1,x_2,y_2,\dots]$ be the polynomial ring in $x$ and $y$ together.

\begin{definition}
[{\cite[Proposition 2.4.7]{Manivel}}]
\label{doubleschub-def} The \emph{double Schubert polynomial} of $w \in S_\infty$ is 
\[ \fkS_w(x;y) \omdef= \sum_{\substack{w=v^{-1}u \\  \ell(w) = \ell(u)+\ell(v)}}  \fkS_u(x) \fkS_{v}(-y) \in \cP_\infty(x;y).
  \]
  \end{definition}

  Let $S_\infty$ act on $\cP_\infty(x;y)$ by permuting only the $x_i$ variables, and extend the formula for 
   $\partial_i$   to an operator $\cP_\infty(x;y) \to \cP_\infty(x;y)$ with respect to this action. The following then holds:
  
    \begin{theorem}[{See \cite[\S2.3.1]{Manivel}}]
  If
   $n \in \PP$ and $v \in S_n$ then 
 \[\fkS_{w_n}(x;y) = \prod_{i+j \leq n} (x_i-y_j)
 \qquand
  \fkS_v(x;y) =  \partial_{v^{-1} w_n} \fkS_{w_n}(x;y)\]
  where the product on the left is over $i,j \in \PP$.
  In particular, $\fkS_v = \fkS_v(x;0)$.
\end{theorem}

If $w \in S_n$ then $\fkS_w$ is a polynomial in at most $n-1$ variables, though often fewer.

\begin{proposition}[{\cite[Proposition 2.5.4]{Manivel}}]\label{schubbasis-prop}
The set of Schubert polynomials $\fkS_w$ with $w \in S_\infty$ ranging over all permutations with largest descent at most $n$ forms a basis for $\cP_n$ over $\ZZ$. 

\end{proposition}

As mentioned in the introduction, a permutation is \emph{dominant} if it is 132-avoiding.
Alternatively, a permutation is dominant if and only if its Rothe diagram is the diagram of a partition \cite[Exercise 2.2.2]{Manivel}. 
A permutation $w \in S_n$ is \emph{Grassmannian} if it has at most one right descent, or equivalently if
 for some $r\geq 0$ it holds that $c_1(w) \leq \dots \leq c_r(w)$ and $c_i(w) = 0$ for all $i>r$.
For permutations of these types, the corresponding Schubert polynomials have the following   formulas.
Let $s_\lambda$ denote the Schur function indexed by a partition $\lambda$.

\begin{proposition}[{\cite[Propositions 2.6.7 and 2.6.8]{Manivel}}]\label{schubertfactor-prop}
Let $w \in S_\infty$.
\ben
\item[(a)] If $w$ is dominant then $\fkS_w(x;y) = \prod_{(i,j) \in D(w)} (x_i-y_j)$.

\item[(b)] If $w\neq 1$ is Grassmannian with unique descent $r$, then $\fkS_w = s_{\lambda(w)}(x_1,\dots,x_r)$.
\een
\end{proposition}

\subsection{Cohomology of  flag varieties}\label{cohomology-sect}

We review the geometric context that leads to the consideration of Schubert polynomials.
Let $\Fl(n)$ denote the set of \emph{complete flags}
$F_\bullet = ( 0 = F_0 \subsetneq F_1 \subsetneq \dots \subsetneq F_n = \CC^n),$ where each $F_i$ is a subspace of dimension $i$, given the structure of a projective algebraic variety via the Pl\"ucker embedding as in \cite[\S3.6.1]{Manivel}. We identify $\Fl(n)$ with the right coset space $B \setminus \GL_n(\CC)$, where $B$ is the Borel subgroup of lower triangular matrices in   $\GL_n(\CC)$.

The general linear group $\GL_n(\CC)$ acts on the right on $\Fl(n)$ by multiplication. 
Let $B^+ = w_n\cdot  B\cdot  w_n$ denote the Borel subgroup {opposite} to $B$, consisting  of 
the  upper triangular matrices in $\GL_n(\CC)$. 
It follows from the Bruhat decomposition of $\GL_n(\CC)$  that the distinct orbits of $B^+$ on $\Fl(n)$ are given by  $B\backslash BwB^+$ for $w \in S_n$, where $S_n$ is embedded as the subgroup of permutation matrices in $\GL_n(\CC)$.
Define 
\be\label{mathring-eq} \mathring X_w \omdef= B\backslash BwB^+ \qquand X_w \omdef= \overline{B\backslash BwB^+}\qquad\text{for }w \in S_n,\ee
where on the right the bar denotes the Zariski closure. We call $ \mathring X_w$ the \emph{Schubert cell} attached to $w \in S_n$ and $X_w$ the corresponding \emph{Schubert variety}. 

\begin{remark}\label{B+-rmk}
Because we identify $\Fl(n)$ with $B\backslash \GL_n(\CC)$ rather than $\GL_n(\CC) / B$ and define Schubert cells to be right $B^+$-orbits rather than left $B$-orbits, our definitions differ from those in \cite[\S3.6]{Manivel} by a transformation of indices. Explicitly, 
the sets $\Omega_w$ for $w \in S_n$ which Manivel refers to as  Schubert cells are given in our notation by
$ \Omega_w = \mathring X_{w_nw} \cdot w_n$.
What we call $X_w$ is related  to Manivel's definition of the Schubert variety of $w \in S_n$ by the same transformations. It thus follows
from \cite[\S3.6.2]{Manivel} that $X_w$ is an irreducible variety of codimension $\ell(w)$  in $\Fl(n)$.
The set $\mathrm{Fl}(n)$ is itself
  the Schubert variety indexed by the identity element of $S_n$ in our
  definitions.
\end{remark}

It will be useful to review the following concrete description of Schubert cells and varieties.
Choose a basis $e_1,e_2,\dots, e_n$ of $\CC^n$ for each $j \in [n]$
 define $V_j = \CC\spanning\{e_1,e_2,\dots,e_j\}$. Given a vector space $U\subset \CC^n$, we write $\proj : U \to V_j$  for the restriction to $U$ of the usual  linear projection $\CC^n \to V_j$ mapping $e_i \mapsto 0$ for $i>j$.
Also define  
\be\label{r-def}
\r_w(i,j) \omdef= | \{ t \in [i] : w(t) \in [j]\}|\qquad\text{for $w \in S_n$ and $i,j \in [n]$.}
\ee
By \cite[Proposition 3.6.4]{Manivel} (noting the remark above), we then have
\be\label{rank-cond}
\ba
\mathring X_w &= \left\{ F_\bullet \in \Fl(n) : \dim\( \Proj: F_i \to V_j\) = \r_w(i,j) \text{ for each }i,j \in [n]\right\},
\\
X_w &=  \left\{ F_\bullet \in \Fl(n) : \dim\( \Proj : F_i \to V_j\) \leq  \r_w(i,j) \text{ for each }i,j \in [n]\right\}.
\ea
\ee
These conditions say that $F_\bullet$ belongs to $\mathring X_w$ (respectively, $X_w$) if and only if for each $i,j \in [n]$, the upper left $i \times j$ submatrix of a matrix representing $F_\bullet$ has rank equal to (respectively, at most) the number of 1's in the upper left $i\times j$ submatrix of the permutation matrix of $w$.

If $X$ is a smooth complex algebraic variety and $V$ is a closed subvariety, then there is a corresponding cohomology class $[V] \in H^*(X,\ZZ)$,  with the important property that $[V\cap W] = [V][W]$ when $V$ and $W$ intersect transversely on an open subset of $V \cap W$.
When $X$ is compact one defines $[V]$ by first  triangulating $V$ to obtain a homology class and then taking its Poincar\'e dual. In general, one can view $[V ]$ as the image of the class of $V$ in the Chow ring of $X$ under an appropriate map to $H^*(X)$; see \cite{Fulton1997} or \cite[Appendix A]{Manivel} or \cite[Chapter 19]{Fulton1984}. 

For each Schubert variety $X_w\subset \Fl(n)$ one obtains in this way a corresponding \emph{Schubert class} $[X_w] \in H^*(\Fl(n),\ZZ)$ which is denoted $\sigma_w$ in \cite[\S3.6.3]{Manivel}.
 As in the introduction, we identify the Schubert classes with elements of the coinvariant algebra of the symmetric group via the 
  \emph{Borel isomorphism}
 (see \cite[\S3.6.4]{Manivel}) 
 \be
 \label{borel-eq}
H^*(\Fl(n),\ZZ) \xrightarrow{\sim} \cP_n / (\Lambda_n^+),
\ee
with $(\Lambda_n^+)$ denoting the ideal in $\cP_n$ generated by the symmetric polynomials of positive degree.
Via these identifications, 
the divided differences $\partial_w$ for $w \in S_n$
make sense as operators on $H^*(\Fl(n),\ZZ)$, since \eqref{+eq} implies that $\partial_i$
  maps  $(\Lambda_n^+)$ into itself.
Bernstein, Gelfand, and Gelfand  \cite{BGG} show that 
\be\label{BGG} \partial_s [X_w] = \begin{cases} [X_{ws}] &\text{if }s \in \DesR(w) \\ 0 &\text{otherwise}\end{cases}
\qquad\text{ for $w \in S_n$ and }s \in \{s_i : i \in \PP\}.
\ee
Consequently, once one fixes a polynomial representing $[X_{w_n}]$ (the unique closed orbit, in this case the class of a point),
representatives for all $[X_w]$ are determined by induction. Lascoux and Sch\"utzenberger's work \cite{LS1982b} shows that  the Schubert polynomials are representatives of the Schubert classes formed in precisely this way:

\begin{theorem}[Lascoux and Sch\"utzenberger \cite{LS1982b}] 
\label{schub-thm}
For all $w \in S_n$  it holds that
$\fkS_w \equiv [X_w]$.
\end{theorem}

\subsection{Stanley symmetric functions}
\label{stanley-sect}

Let $\Lambda = \Lambda(x)$ 
be  the algebra of
symmetric functions over $\ZZ$ in the variables $x= \{x_i : i \in \PP\}$. We follow the standard conventions from \cite{EC2} for referring to the various well-known bases of this algebra. 

Our first definition of the {Stanley symmetric function} $F_w$  was given by \eqref{F1-eq}. Stanley \cite{Stan} was the first to consider this power series and prove that it  belongs to $\Lambda$. In this section we review an alternate definition due to Edelman and Greene \cite{EG} which makes this fact more transparent and explains the connection between $F_w$ and the problem of counting reduced words.

  \begin{remark}
  Following Lam \cite{Lam,Lam2}, our conventions for  $F_w$ differ from Stanley's original definition by the transformation $w \leftrightarrow w^{-1}$;  \cite[Corollary 2.2]{Lam} is helpful for understanding these transformations.
   \end{remark}

Let $T$ be a \emph{(Young) tableau}, i.e., an assignment of positive integers to the cells of the diagram of a partition (or, more generally, to the cells of some sequence of partitions or skew shapes or shifted shapes),
called the \emph{shape} of $T$. 
Say that $T$ is \emph{strict} if its entries are   strictly increasing both from left to right in each row and from top to bottom in each column. A strict tableau is \emph{standard} if its entries comprise the set $[n]$ for some $n \in \NN$. 
The \emph{reverse reading word} of $T$, denoted $\rww(T)$,   is the word obtained by reading the rows of $T$ from right to left, starting with the top row. For example, 
\[
T =  {\tableau[s]{1 &2 &3  \\ 2 & 3 }}
 \]
 has $\rww(T) = (3,2,1,3,2)$. 
 A tableau is \emph{reduced} for $w \in S_\infty$ if it is strict and its reverse reading word is a reduced word for $w$, where we identify a sequence of positive integers $(i_1,i_2,\dots,i_k)$ with the word $(s_{i_1}, s_{i_2},\dots, s_{i_k})  $.
 The tableau $T$ above
 is reduced  for  $w= s_3s_2s_1s_3s_2 = [4,3,1,2]$. 
Results of \cite{EG} show that we may alternatively define the Stanley symmetric function $F_w$ as follows:
  
\begin{theorem}[Edelman and Greene \cite{EG}]\label{stanley-def} If $w \in S_\infty$ then $F_w = \sum_{\lambda} \alpha_{w,\lambda} s_\lambda \in \Lambda$
where the sum is over  partitions $\lambda$  and
  $\alpha_{w,\lambda}$ is the number of reduced tableaux for $w$ of shape $\lambda$. 
  \end{theorem}

Let $f^\lambda$ denote the number of standard tableaux of shape $\lambda$.

\begin{theorem}[Edelman and Greene \cite{EG}] \label{EG-cor} If $w \in S_\infty$ then $|\cR(w)| = \sum_{\lambda} \alpha_{w,\lambda} f^\lambda$.
\end{theorem}

Edelman and Greene \cite{EG} provide  bijective proofs of these identities using a variant of the RSK correspondence, now referred to as \emph{Edelman-Greene insertion}. This map gives an algorithm for calculating $F_w$ for any $w \in S_\infty$; other, more efficient methods of computation are described in \cite{Garsia,HY,Lascoux,Little}.
For our purposes, it will  suffice   to recall one exact formula.

\begin{proposition}[{Billey, Jockusch, Stanley \cite[Proposition 2.4]{BJS}}]\label{Fskew-prop} If the Rothe diagram of $w \in S_n$ is equivalent to a skew shape $\lambda/\mu$, then  $F_w = s_{\lambda/\mu}$ is the corresponding skew Schur function.
\end{proposition}

A permutation $w \in S_n$ is \emph{vexillary} if it is 2143-avoiding or, equivalently,  if  its Rothe diagram is equivalent to the  diagram of a partition \cite[Proposition 2.2.7]{Manivel}.

\begin{theorem}[{Macdonald \cite[Eq.\ (7.24)(iii)]{Macdonald}; Stanley \cite[Theorem 4.1]{Stan}}] \label{Mac-thm}
 The Stanley symmetric function $F_w$ is a Schur function if and only if $w$ is vexillary, in which case  $F_w = s_{\lambda(w)}$.
\end{theorem}
 
 \begin{example} The longest permutation $w_n \in S_n$ is vexillary with  $\lambda(w_n) = \delta_n$, so  $F_w = s_{\delta_n}$ and  via Theorem \ref{EG-cor} we recover  the result of Stanley \cite{Stan} that $|\cR(w_n)|=f^{\delta_n}$.
 \end{example}

\section{Involution words for symmetric groups}\label{results-sect}

As in Section \ref{gen-sect},  we let 
$\I(S_\infty)$ and $ \I(S_n) $ be the sets of involutions in $S_\infty$ and $S_n$.
In addition, define  $\Ifpf(S_{2n})$ to be the set of fixed-point-free involutions in $S_{2n}$, and set $\Ifpf(S_\infty) = \bigcup_{n\in \PP} \Ifpf(S_{2n})$. 
Throughout this section, we write 
 $\cG_n $ for the Grassmannian involution
\be\label{g-def} \cG_n \omdef= (1,n+1)(2,n+2)\cdots(n,2n) = [n+1,n+2,\dots,2n,1,2,\dots,n]\in \Ifpf(S_{2n}).\ee
As in Section~\ref{gen-sect}, let  $\kappa(y)$
be the number of  2-cycles in an involution $y$.
We   have
\[
\ellhat(y) = \tfrac{1}{2} \( \ell(y)+ \kappa(y)\)
\qquand
\ellhat_\fpf(z) = \tfrac{1}{2} \( \ell(z)- \kappa(z)\)
\]
for  $y \in \I(S_n)$ and $z \in \Ifpf(S_{2n})$,
where $\ellhat(y)$ is as in \eqref{ellhat-def} and 
$ \ellhat_\fpf(z)\omdef =  \ellhat(z) - n.$

\subsection{Atoms for permutations}
\label{Symatoms-sect}

In this section we discuss some properties of the sets
 $\cA(y,z)$ from Theorem-Definition \ref{atoms-thmdef}, in the special case when $y,z \in \I(S_\infty)$. As in the introduction, we write  $\cAfpf(z) = \cA(\wfpf_n,z)$ for $z \in \Ifpf(S_{2n})$
 where $\wfpf_n = (1,2)(3,4)\cdots(2n-1,2n)$,
 so that $\cRfpf(z) = \bigcup_{u \in \cAfpf(z)} \cR(u)$.

As noted earlier, 
Can, Joyce, and Wyser \cite{CJ,CJW} have recently studied
the sets  $\cA(y)$ and $\cAfpf(y)$ for  $y \in \I(S_\infty)$; they 
 provide a useful set of conditions, involving only the one-line representations of permutations, that classify their elements. (Several left/right-handed conventions in \cite{CJ,CJW} are the mirror images of the ones we adopt here, so the elements in the sets described by the 
 main results  \cite[Theorem 2.5 and Corollary 2.16]{CJW} are actually the inverses of  
 what we call atoms.)
The sets $\cA(y)$ and $\cAfpf(y)$   may also be viewed as special cases of the sets $W(Y)$ that Brion defines geometrically in \cite[\S1.1]{Brion98}.
These sets of atoms  have several special 
  properties which do not generalize to other Coxeter groups, which we discuss in the complementary paper \cite{HMP2}:

\begin{theorem}[{\cite[Corollaries 6.11 and 6.23]{HMP2}}] \label{atomicSn-thm}
Let $y \in \I(S_\infty)$. Then  $|\cA(y)| = 1$  if and only if 
$y$ is 321-avoiding.
Likewise, if $y$ is fixed-point-free, then
 $|\cAfpf(y)| = 1$  if and only if 
$y$ is 321-avoiding.
\end{theorem}
We can identify one atom of any involution $y$  by the following  construction. 
Given a list  $[c_1,c_2,\dots]$, we write $[[c_1,c_2,\dots]]$ for the sublist formed by omitting  repeated entries after their initial occurrence.
For example, $[[1,1,3,4,2,4,3,2,4]] = [1,3,4,2]$. 
If  the set of distinct elements
in $[c_1,c_2,\dots]$ is $\{1,2,\dots,n\}$, then interpret $[[c_1,c_2,\dots]]$ as the one-line representation of a permutation in $S_n$.
Now, for   $y  \in \cI(S_n)$, we   define $\alpha_{\min}(y)$ and $\beta_{\min}(y)$ to be the permutations in $S_n$ given by
\be\label{min-atom-eq}
\alpha_{\min}(y) \omdef= [[ b_1,a_1,b_2,a_2,\dots,b_k,a_k]]^{-1}
\qquand
\beta_{\min}(y) \omdef= [[a_1,b_1,a_2,b_2,\dots,a_k,b_k]]^{-1}
\ee
where $(a_1,b_1),(a_2,b_2),\dots,(a_k,b_k)$ are the elements of the set $\{ (a,b) \in [n]\times [n] : a \leq b = y(a)\},$ indexed such that $a_1<a_2<\dots<a_k$.

\begin{example}\label{alphabeta-ex}
If $y =  [4,7,3,1,6,5,2] =(1,4)(2,7)(5,6) $ 
then
\[ \alpha_{\min} (y)= [[4,1,7,2,3,3,6,5]]^{-1} = [2,4,5,1,7,6,3]
\qquand
 \beta_{\min} (y)= [1,3,5,2,6,7,4].
\]
In turn, one computes that
\ben
\item[(a)] $\alpha_{\min}(\cG_k) = [2,4,\dots,2k,1,3,\dots,2k-1]$ and $\beta_{\min}(\cG_k) = [1,3,\dots,2k-1,2,4,\dots,2k]$.

\item[(b)] 
$ \alpha_{\min}(w_{2k}) = [2,4,\dots,2k,2k-1,\dots,3,1]
$
and
$
\beta_{\min}(w_{2k}) = [1,3,\dots,2k-1,2k,\dots,4,2]
$.
\een
\end{example}

The following proposition from \cite{HMP2} is a corollary of  results in   \cite{CJW}.

\begin{proposition}[{\cite[Theorems 6.10 and 6.22]{HMP2}}] \label{lexatom-prop}
If $y \in \I(S_\infty)$ and  $z \in \Ifpf(S_\infty)$
then $\alpha_{\min}(y)$ and $\beta_{\min}(z)$ are the lexicographically minimal elements of $\cA(y)$ and $\cAfpf(z)$, respectively.
\end{proposition}

\begin{corollary}\label{lexatom-cor}
If $y \in \I(S_\infty)$ (respectively, $z \in \Ifpf(S_\infty)$) is 321-avoiding, then so is $\alpha_{\min}(y)$ (respectively, $\beta_{\min}(z)$).
\end{corollary}

\begin{proof}
The set of 321-avoiding permutations is an order ideal under the left weak order (see \cite[Proposition 2.4]{Stem}), and  we have $u <_L \pi$ whenever $u \in \cA(\sigma,\pi)$ by inspection.
\end{proof}

\subsection{Diagrams and codes for involutions}\label{invdiagram-sect}

For involutions $y \in \I(S_\infty)$, we define
\be\label{invol-rothe-eq} \ba
\D(y) &\omdef= \{ (i,j) \in \PP\times \PP :  j<y(i)\text{ and } i<y(j)\text{ and }j\leq i\},
\\
\Dfpf(y) &\omdef=  \{ (i,j) \in \PP\times\PP :  j<y(i)\text{ and } i<y(j)\text{ and }j<i\}.
\ea
\ee
Call these sets \emph{involution Rothe diagrams}.
Observe that $\D(y)$ and $\Dfpf(y)$ 
are  the subsets of positions in the usual Rothe diagram $D(y)$ that are weakly and strictly below the diagonal.
Since 
$D(y)$ is invariant under transpose as $y^2=1$,   the diagram $\D(y)$  uniquely determines $ D(y)$.

\begin{example}\label{invRothe-ex}
If $y =  [4,7,3,1,6,5,2]$
then 
\[ D(y) =
 \left\{
\barr{cccccc}
\circ & \circ & \circ & . & . & . 
\\
\circ & \circ & \circ & . & \circ & \circ 
\\
\circ & \circ & . & . & . & . 
\\
. & . & . & . & . & . 
\\
. & \circ & . & . & \circ & . 
\\
. & \circ & . & . & . & . 
\earr
\right\}
\qquand
\hat D(y) = 
 \left\{
\barr{cccccc}
\circ & . & . & . & . & . 
\\
\circ & \circ & . & . & . & . 
\\
\circ & \circ & . & . & . & . 
\\
. & . & . & . & . & . 
\\
. & \circ & . & . & \circ & . 
\\
. & \circ & . & . & . & . 
\earr
\right\}.
\]
One similarly computes that
\[
\hat D(\cG_n) = \{ (i,j)  : 1\leq j\leq i \leq n\}
\qquand
\hat D(w_n) = \{ (i,j) \in \PP\times \PP: i+j \leq n \text{ and } j\leq i\},
\]
which are the  transposes of the
 shifted shapes of $ (n,n-1,\dots,2,1)$ and  $(n-1,n-3,n-5,\dots)$.
\end{example}

We discuss a few results that indicate why $\hat D(y)$ and $\Dfpf(z)$ are the appropriate notions of diagrams for involutions. 
The cardinality of $D(w)$ is the number of inversions of $w \in S_\infty$, and so $\ell(w) = |D(w)|$.
An analogous fact holds for involution Rothe diagrams:

\begin{proposition} If $y \in \I(S_n)$ and $z \in \Ifpf(S_{2n})$ then
$
\ellhat(y) = |\hat D(y)| $ and $ \ellhat_\fpf(z) = |\Dfpf(z)|
.
$
\end{proposition}

\begin{proof}
When $y$ is an involution, $D(y)$ is transpose-invariant and the number of diagonal positions $(i,i) \in D(y)$ is precisely $\kappa(w)$,   so it follows that
$|\hat D(y)| = \tfrac{1}{2}\(|D(y)| - \kappa(y)\) + \kappa(y) = \ellhat(y).$ If $z$ is a fixed-point-free involution, then $|\Dfpf(z)| =|\hat D(z)| - \kappa(z) = \ellhat(z) - \kappa(z) = \ellhat_\fpf(z)$.
\end{proof}

Given $y \in \I(S_n)$,
we define the \emph{involution codes}
\be\label{inv-code-eq}
\hat c(y) \omdef=\(\hat c_1(y),\hat c_2(y),\dots,\hat c_n(y)\)   \quand \hat c_\fpf(y)\omdef=\(\hat c_{\fpf,1}(y),\hat c_{\fpf,2}(y),\dots,\hat c_{\fpf,n}(y)\) 
\ee
to be the integer sequences with
\[
\ba
 \hat c_i(y) &= | \{ j \in [n] :  y(j)\leq i< j \text{ and }y(i) > y(j)\}|,
\\
 \hat c_{\fpf,i}(y) &= | \{ j \in [n] :  y(j)< i< j \text{ and }y(i) > y(j)\}|.
\ea
\]
The $i^{\mathrm{th}}$ entries in these sequences count
the number of cells in the $i^{\mathrm{th}}$ rows of  $\hat D(y)$ and $\Dfpf(y)$, respectively.
These sequences do not depend in any serious way on $n$: if $y $ is viewed as belonging to a larger symmetric group $S_N \supset S_n$, then the resulting codes are the same, extended by zeros.
\begin{example}\label{cg-ex} If $y=  [4,7,3,1,6,5,2]=(1,4)(2,7)(5,6) \in \I(S_7)$ then
\[\hat c(y) = (1,2,2,0,2,1,0)
\qquand
\hat c_\fpf(y) = (0,1,2,0,1,1,0).
\]
 For the involutions $\cG_n, w_{2n} \in \Ifpf(S_{2n})$
we have:
 \ben
\item[(a)] $\hat c(\cG_n) = (1,2,3,\dots,n,0,\dots,0)$ and $\hat c_\fpf(\cG_n) = (0,1,2,\dots,n-1,0,\dots,0)$.
\item[(b)] $\hat c(w_{2n}) = (1,2,\dots,n-1,n,n-1,\dots,2,1,0)$ and $\hat c_\fpf(w_{2n}) = (0,1,2,\dots,n-1,n-1,\dots,2,1,0)$.
\een
\end{example}
Recall the  minimal atoms $\alpha_{\min}(y) \in \cA(y)$ and $\beta_{\min}(y) \in \cAfpf(y)$ defined by \eqref{min-atom-eq}.

\begin{lemma}\label{invcode-lem}
If $y \in \I(S_n)$   then
$\hat c(y) = c( \alpha_{\min}(y))$ and $ \hat c_\fpf(y) = c(\beta_{\min}(y)).
$
\end{lemma}

\begin{proof}
Define $(a_i,b_i)$ relative to $y$ as before Example \ref{alphabeta-ex} and 
fix $i \in [k]$, where $k = n-\kappa(y)$ is the size of $\{ (a,b) \in [n]\times [n] : a \leq b = y(a)\}$.
It is straightforward to check from the definitions that
\[ c_{a_i}(\beta_{\min}(y)) = | \{ j \in [i-1] :  b_j > a_i \}|
\qquand
c_{b_i}(\beta_{\min}(y)) = | \{ j \in [i-1] :  b_j > b_i \}|.
\]
On the other hand, since $y(t) \leq t$ implies $y(t) = a_j$ and $t=b_j$ for some $j$, and $a_j < a_i $ if and only if $j < i$, it follows that
\[ \ba \hat c_{\fpf,a_i}(y) &=|\{ j \in [n] : y(j)< a_i < j \text{ and }y(a_i) > y(j) \}|
\\&=|\{ j \in [k] : a_j< a_i < b_j \text{ and }b_i > a_j \}|
=|\{ j \in [i-1] : b_j>a_i \}|  =  c_{a_i}(\beta_{\min}(y))
\ea\]
and likewise
\[ \ba \hat c_{\fpf,b_i}(y) &=|\{ j \in [n] : y(j)< b_i < j \text{ and }y(b_i) > y(j) \}|
\\&=|\{ j \in [k] : a_j< b_i < b_j \text{ and }a_i > a_j \}|
=|\{ j \in [i-1] : b_j>b_i \}|
= c_{b_i}(\beta_{\min}(y))
.
\ea\]
As every $t \in [n]$ is given by $a_j$ or $b_j$ for some $j \in [k]$, we conclude that $\hat c_\fpf(y) = c( \beta_{\min}(y))$.
The proof that $\hat c(y) = c( \alpha_{\min}(y))$ is similar; we omit the details.
\end{proof}

One can read off an involution word for $y \in \I(S_n)$ from its involution code in the following way.
For any sequence $c=(c_1,c_2,\dots,c_n) \in \NN^n$, define 
\[\Theta(c) \omdef=  (\underbrace{c_1,\dots, 2,1}_{c_1\text{ terms}},\underbrace{c_2 + 1,\dots, 3,2}_{c_2\text{ terms}},\underbrace{c_3 + 2,\dots, 4,3}_{c_3\text{ terms}},\dots,\underbrace{c_n + n-1,\dots, n+1,n}_{c_n\text{ terms}}).\]
In what follows,  identify  integer sequences $(i_1,\dots,i_k)$  with words $(s_{i_1},\dots,s_{i_k})$.
\begin{proposition} If $y \in \I(S_n)$ and   $z \in \Ifpf(S_{2n})$ 
then
$\Theta\(\hat c(y)\) \in \hat\cR(y)$
and
 $\Theta\(\hat c_\fpf(z) \) \in \cRfpf(z).$
\end{proposition}

\begin{proof}
By \cite[Remark 2.1.9]{Manivel}, if $u$ is any permutation then $\Theta(c(u))$ gives a reduced word for $u$, so this result follows from Lemma \ref{invcode-lem}.
\end{proof}

\subsection{Involution Schubert polynomials}
\label{invschubert-sect}

In this section we turn to the involution Schubert polynomials $\fkS_y$ and $\Sfpf_y$ defined in the introduction \eqref{intro-inv-schub-eq}.
For  $y,z \in \I(S_\infty)$,
we have $\hat\fkS_{y,z} = \sum_{u \in \cA(y,z)} \fkS_u \in \cP_\infty$
and
$
\hat \fkS_y = \hat \fkS_{1,y} = \sum_{w \in \cA(y)} \fkS_w$.
For $z \in \Ifpf(S_{2n})$ we have
$
 \Sfpf_z = \hat \fkS_{\wfpf_n,z} = \sum_{w \in \cAfpf(z)} \fkS_w$
 where $\wfpf_n = s_1s_3\cdots s_{2n-1}$.

Write $\Lambda_n$ for the  ring of symmetric polynomials in $\cP_n = \ZZ[x_1,\dots,x_n]$.
For the Grassmannian involution $\cG_k $ given by \eqref{g-def} we can compute $\hat \fkS_{\cG_k}$ and $\Sfpf_{\cG_k}$ directly:
\begin{proposition}\label{grassman-prop}
For each $k \in \NN$ it holds that 
\[ \hat\fkS_{\cG_k} =2^{-k}\prod_{1 \leq i \leq j \leq k} (x_i+x_j) \in \Lambda_k \qquand \Sfpf_{\cG_k}  = \prod_{1 \leq i < j \leq k} (x_i+x_j) \in \Lambda_k. \]
\end{proposition}

\begin{proof}
Combining Proposition \ref{schubertfactor-prop} with Theorem~\ref{atomicSn-thm}, Proposition~\ref{lexatom-prop}, and Example~\ref{cg-ex}(a)
shows that $\hat\fkS_{\cG_k} = s_{\delta_{k+1}}(x_1,\dots,x_k)
$ and $\Sfpf_{\cG_k} = s_{\delta_{k}}(x_1,\dots,x_k).$
The product formulas given for these Schur polynomials follow from Jacobi's determinantal definition of $s_\lambda$; see \cite[Exercise 1.2.4]{Manivel}.
\end{proof}

The involution Schubert polynomials \eqref{intro-inv-schub-eq} may also be characterized along the lines of Theorem \ref{schubunique-thmdef}, as follows.
We write $\delta_{v,w}$ to the denote the   Kronecker delta function, equal to 1 if $v=w$ and 0 otherwise.
The two-sided weak order $<_T$ in the following theorem is defined by \eqref{t-eq}.

\begin{theorem}\label{invschubdef-thm}
Fix $y \in \I(S_\infty)$. Then $\{ \hat\fkS_{y,z} \}_{z \in \I(S_\infty)}$ is the unique family of homogeneous polynomials indexed by $\I(S_\infty)$ such that 
\[\hat\fkS_{y,z} =\delta_{y,z}\text{ if $y \not <_T z$}
\qquand
\partial_i \hat\fkS_{y,z} = \begin{cases} \hat\fkS_{y,z\rtimes s_i}&\text{if $s_i \in \DesR(z)$} \\0&\text{otherwise}\end{cases}\quad\text{for all $i \in \PP$.}
\]
\end{theorem}

\begin{proof}
We first claim that $\hat\fkS_{y,z}$ has  these properties. 
The polynomials $\hat\fkS_{y,z}$ are homogeneous since the Schubert polynomial $\fkS_u$ is homogeneous of degree $\ell(u)$.  
That $\hat\fkS_{y,z} =\delta_{y,z}$ if $y\not <_Tz$ follows from the definition of $\cA(y,z)$.
Finally, 
the given formula for $\partial_i \hat \fkS_{y,z}$  is straightforward to check from our original definition \eqref{intro-inv-schub-eq} using Proposition \ref{atomdes-prop} and Theorem \ref{schubunique-thmdef}.

For the uniqueness assertion, suppose $\{ f_z \}_{z \in \I(S_\infty)}$
is another family of homogeneous polynomials with the properties of $\hat \fkS_{y,z}$ described in the theorem.
We proceed as in the proof of \cite[Theorem 2.3]{Knutson}.
By hypothesis $f_z = \hat \fkS_{y,z} = \delta_{y,z}$  if $y \not <_Tz$, so assume that $y <_T z$ and that $f_u=\hat \fkS_{y,u}$  if $u \in \I(S_\infty)$ is such that $\ellhat(u) < \ellhat(z)$.
Then $\partial_i f_z = \partial_i \hat \fkS_{y,z}$ for all $i \in \PP$, so   we deduce that $f_z =\hat \fkS_{y,z}+a_y$ for some  $a_y \in \ZZ = \bigcap_{i \in \PP} \ker \partial_i $. 
Since $f_z$ and $\hat \fkS_{y,z}$ are both homogeneous and since $\hat \fkS_{y,z}$ has  degree $\ellhat(y,z)>0$, the  constant $a_z$ must be zero  so $f_z=\hat \fkS_{y,z}$ as desired.
\end{proof}

By induction,  we may express $\hat \fkS_{y,z}$ in terms of divided differences in the following way.

\begin{corollary}\label{invschubdef-cor} Let $y,z \in \I(S_n)$. Then $\hat\fkS_{y,z} = \partial_{u} \hat\fkS_{y,w_n}$ for any $u \in \cA(z,w_n)$.
\end{corollary}

For technical reasons we need a slightly different version of Theorem \ref{invschubdef-thm} to characterize the fixed-point-free involution Schubert polynomials $\Sfpf_z$.
Define $\tilde \I_\fpf$ to be the set of fixed-point-free involutions $w: \PP \to \PP$ with the property that, for some sufficiently large positive integer $N$, it holds that $w(2i-1) = 2i$ and $w(2i) = 2i-1$ for all $i>N$. 
If $n$ is finite and $z \in \Ifpf(S_{2n})$ then let $ z_\infty$ be the infinite product
\be\label{infprod} z_\infty = z \cdot s_{2n+1}\cdot   s_{2n+3} \cdot s_{2n+5} \cdots \in \tilde\I_\fpf.\ee
We also let $1_\fpf = 1_\infty \in \tilde \I_\fpf$ be the map $\PP\to \PP$ with $2i-1 \mapsto 2i$ and $2i\mapsto 2i-1$ for all $i$. 
The set $\tilde \I_\fpf$ may be equivalently defined as the orbit of $1_\fpf$ under the conjugation action of $S_\infty$ on the group of all bijections $\PP \to \PP$. 
Although no elements of $\tilde\I_\fpf$ belong to $S_\infty$,
we define $w\rtimes s_i$ for $w \in \tilde\I_\fpf$ exactly as if $w$ were in $S_\infty$, and we define the (now infinite) set $\DesR(w)$ again by \eqref{Des-def}.

If $z \in \Ifpf(S_{2n})$, so that $zs_{2n+1} \in \Ifpf(S_{2n+2})$,
then
it follows either as a straightforward exercise or from the more general statement \cite[Lemma 3.2]{HMP2} that
$\cA(\wfpf_n, z) = \cA(\wfpf_{n+1}, z s_{2n+1})$ and therefore $
\Sfpf_z= \Sfpf_{z s_{2n+1}}$. By induction,  if $y \in \Ifpf(S_{2m})$ and $z \in \Ifpf(S_{2n})$, then $\Sfpf_y = \Sfpf_z$ whenever $ y_\infty =  z_\infty$.
It is therefore well-defined to set 
\[ \Sfpf_{ z_\infty} = \Sfpf_z = \hat\fkS_{\wfpf_n,z} \qquad\text{for any $z \in \Ifpf(S_{2n})$ and $n \in \PP$.}\]
Since $\tilde \I_\fpf = \bigcup_{n \in \PP} \{  z_\infty : z \in \Ifpf(S_{2n})\}$, this defines $\Sfpf_w$ for every $w \in \tilde \I_\fpf$.

\begin{corollary}\label{invschubdeffpf-cor}
It holds that  $\{ \Sfpf_{z} \}_{z \in \tilde\I_\fpf}$ is the unique family of homogeneous polynomials indexed by $\tilde\I_\fpf$ such that 
\[\Sfpf_{1_\fpf} =1
\qquand
\partial_i \Sfpf_{z} = \begin{cases} \Sfpf_{z\rtimes s_i}&\text{if $s_i \in \DesR(z)$ and $z\rtimes s_i \in \tilde \I_\fpf$} \\0&\text{otherwise}\end{cases}
\quad\text{for all $i \in \PP$.}
\]
\end{corollary}

\begin{proof}
If   $z \in \I(S_{2n})$ is not fixed-point-free then it cannot hold that $\wfpf_n \leq_T z$ (since $y\leq_T z$ implies that $y$ has at least as many fixed points as $z$) so $\hat\fkS_{\wfpf_n,z} = 0$. From this observation, the assertion that the polynomials $\{ \Sfpf_{z} \}_{z \in \tilde\I_\fpf}$ have the given properties is an easy exercise from Theorem \ref{invschubdef-thm}.
The   uniqueness assertion  follows by the same argument as   the one in the proof of Theorem \ref{invschubdef-thm}, \emph{mutatis mutandis}, after replacing the involution length $\ellhat$ which is not well-defined on $\tilde\I_\fpf$ by the function $\ellhat_\fpf : \tilde \I_\fpf \to \NN$ given by $\ellhat_\fpf(z_\infty) = \ellhat(\wfpf_n,z) = \ellhat(z) - n$ for $z \in \Ifpf(S_{2n})$.
\end{proof}

Given a sequence $i = (i_1,i_2,\dots,i_n) \in \NN^n$, we let $x^i = x_1^{i_1}x_2^{i_2} \cdots x_n^{i_n}$ and 
write $x^i <_{\mathrm{lex}} x^j$ when $i <_{\mathrm{lex}} j \in \NN^n$, where $<_{\mathrm{lex}}$ denotes the  lexicographic order on sequences.
A different convention is sometimes taken to define lexicographic order on monomials, as we explain in the following remark.

\begin{remark}
Viewing $x^i$ as a word in the alphabet $\{x_1,x_2,\dots\}$
 defines a sequence
\[ \Psi(x^i) \omdef= (\underbrace{1,\dots,1}_{\text{$i_1$ times}},\underbrace{2,\dots,2}_{\text{$i_2$ times}},\dots,\underbrace{n,\dots,n}_{\text{$i_n$ times}}).\]
For example, $\Psi\(x^{(1,1,1)}\) = \Psi(x_1x_2x_3) = (1,2,3)$ and $\Psi\(x^{(0,2,1)}\) = \Psi\(x_2^2x_3\) = (2,2,3)$. One checks that on the set of monomials of any fixed degree,  
the map $x^i \mapsto \Psi(x^i)$ reverses lexicographic order; e.g., we have $x_2^2x_3 <_{\mathrm{lex}} x_1x_2x_3$ but $(1,2,3) <_{\mathrm{lex}} (2,2,3)$. 
 For this reason, the order $<_{\mathrm{lex}}$ that we have defined on monomials is sometimes (e.g., in \cite{BB}) referred to as the \emph{reverse lexicographic order}, though for us this order is the usual \emph{lexicographic total order}.
  \end{remark}

We now show how to read the involution codes $\hat c(y)$ and $\hat c_\fpf(z)$ from the polynomials $\hat \fkS_y$ and $\Sfpf_z$. As these codes   determine  $y \in \I(S_\infty)$ and $z \in \Ifpf(S_\infty)$,    one  can reconstruct the index of an involution Schubert polynomial from the polynomial itself.

\begin{proposition} Let $y \in \I(S_\infty)$ and $z \in \Ifpf(S_{\infty})$.
The lexicographically least monomials in the involution Schubert polynomials $\hat \fkS_y$  and $\Sfpf_z$ are    $x^{\hat c(y)}$
and $x^{\hat c_\fpf(z)}$ respectively.
\end{proposition}

\begin{proof}
One checks that if $u,u' \in S_\infty$ with  $u <_{\mathrm{lex}} u'$ (interpreting $u$ and $u'$  in one-line notation), then $c(u)  <_{\mathrm{lex}} c(u')$.
It therefore suffices by Proposition \ref{lexatom-prop} and Lemma \ref{invcode-lem} to show that $x^{c(u)}$ is the lexicographically least monomial in $\fkS_u$. 
This property is 
clear  
from the proof of \cite[Corollary 3.9]{BB}, on noting  that $x^{c(u)} = x_{D_{\mathrm{bot}}(u)}$ in the notation of \cite[\S3]{BB}, and that  $<_{\mathrm{lex}}$ is what  the authors of \cite{BB} call the reverse lexicographic order on monomials.
\end{proof}

\subsection{Cohomology of flag varieties revisited}\label{revisit-sect}

Throughout this section, we let 
 $\beta$ be a non-degenerate  bilinear form on $\CC^n$ which is symmetric or skew-symmetric,
 and define  $K\subset \GL_n(\CC)$ to be the 
  subgroup of matrices preserving $\beta$.
The group $K$ is given by the orthogonal group $\O(n)$ when $\beta$ is symmetric,
and by the symplectic group $\Sp(n)$ when $\beta$ is skew-symmetric (which can only occur if $n$ is even).

As explained in \cite[\S10]{R}, the orbits of the symmetric subgroups $\O(n)$ and $\Sp(n)$ on $\Fl(n)$ are naturally indexed by $\I(S_n)$ and $\Ifpf(S_n)$, respectively.
 To refer to these indexing sets, we define
 \[ \I_{\O(n)} \omdef= \I(S_n)\qquand \I_{\Sp(n)} \omdef= \Ifpf(S_n).\]
 The corresponding $K$-orbits may then be described by rank conditions analogous to the ones \eqref{rank-cond} for Schubert varieties. Explicitly, we may define a $K$-orbit associated to an involution $y \in \I_K$ by
\[
\mathring Y_y^K \omdef=\left \{F_{\bullet} \in \Fl(n) : \rank \(\beta|_{F_i \times F_j}\) = \r_y(i,j) \text{ for each $i,j \in [n]$}\right\},
\]
where $\beta|_{F_i \times F_j}$ denotes the  linear map $F_i \to F_j^*$ given by 
$v \mapsto \beta(v, \cdot)$, and $\r_y(i,j)$ is as in \eqref{r-def}.
It is not hard to see that $\mathring Y^K_y$ is $K$-stable, and that $\Fl(n)$ is the disjoint union $\bigcup_{y \in \I_K} \mathring Y_y$. Wyser discusses why $\mathring Y_y^K$ is actually a single $K$-orbit in \cite[\S2.1.2]{Wyser}. 

Let $Y^K_y$  denote the Zariski closure of $\mathring  Y^K_y$ as in Section \ref{cohomology-sect}. Then $Y^K_y$ is again defined by rank conditions; namely (see \cite[Proposition 2.4]{Wyser})
\be
\label{Y-eq}
 Y^K_y =\left \{F_{\bullet} \in \Fl(n) : \rank \(\beta|_{F_i \times F_j}\) \leq \r_y(i,j) \text{ for each $i,j \in [n]$}\right\}.
\ee
Equation \eqref{rank-cond} asserts that the Schubert variety $X_w$ for $w \in S_n$ is  the set of complete flags $F_\bullet$ such that $ \dim \(\Proj F_i \to V_j \)\leq \r_w(i,j)$ for each $(i,j) \in [n]\times [n]$.
Fulton shows in \cite{fulton-double-schubert} that a proper subset of these rank conditions  actually imply all of the rest. Specifically, to determine $X_w$ one only needs  the conditions corresponding to pairs $(i,j)$ in the \emph{essential set} of the Rothe diagram $D(w)$. In general, the essential set of a diagram $D \subset \PP \times \PP$ is
\be\label{ess-eq}
\Ess(D) \omdef= \{(i,j) \in D : (i+1,j), (i,j+1), (i+1,j+1)  \notin D \}.
\ee
Observe that $\Ess(D)$ is the set of southeast corners of the connected  components of $D$.
\begin{example}
If $y=[4,7,3,1,6,5,2]$ as in Example \ref{invRothe-ex}
then 
\[\Ess( D(y)) = \{  (2,3),(2,6),(3,2),(5,5),(6,2)\}
\quand
\Ess(\hat D(y))  = \{ (3,2),(5,5),(6,2)\}
.
\]
\end{example}

The rank conditions \eqref{Y-eq} giving $Y^K_y$ admit the following analogous simplification.

\begin{proposition} \label{prop:essential-set} 
The $K$-orbit closure associated to an involution $y \in \I_K$ satisfies
\[
Y^K_y =\left \{F_{\bullet}  \in \Fl(n) : \rank \(\beta|_{F_i \times F_j}\) \leq \r_y(i,j) \text{ for each $(i,j) \in \Ess\(D\)$}\right\}
\]
where $D=\hat D(y)$ when $K=\O(n)$ and $D=\Dfpf(y)$ when $K=\Sp(n)$.
\end{proposition}

\begin{proof}
If $K = \O(n)$ and $y = 1 \in \I(S_n)$, or if $n$ is even and $K = \Sp(n)$ and $y = \wfpf_n \in \Ifpf(S_n)$,
then $\Ess(D) = D = \varnothing$ and the proposition asserts that $Y_y^K = \Fl(n)$.
In the symmetric case when $y=1$, we have $\r_y(i,j) = \min\{i,j\}$ for all $i,j \in [n]$, so the claim that $Y_y^K = \Fl(n)$ is evident from \eqref{Y-eq}.
In the skew-symmetric case when $y=\wfpf_n$,
we have $\r_y(i,j) = \min\{i,j\} - \delta$ where $\delta \in \{0,1\}$ is nonzero if and only if $i=j$ is odd. 
In this case,
since the rank of $\beta|_{F_i\times F_j}$ must be even when $i=j$ and $\beta$ is skew-symmetric, 
it again follows from \eqref{Y-eq} that $Y_y^K = \Fl(n)$ as desired.

Assume $y \in \I_K$ is not one of the trivial involutions considered in the previous paragraph, so that $\Ess(D)$ is nonempty.
Let $C_{ij}$ denote the  set of  complete flags $F_\bullet \in \Fl(n)$ satisfying the condition $\rank \(\beta|_{F_i \times F_j}\) \leq \r_y(i,j)$.  
Since $y$ is an involution and since $\beta$ is symmetric or skew-symmetric,
we have $C_{ij}=C_{ji}$, so \eqref{Y-eq} implies that $Y^K_y$ is the intersection of the sets $C_{ij}$ for   $1\leq j \leq i \leq n$.
Define the \emph{implication graph} of $y$ to be the directed graph on  $\{(i,j) : 1 \leq j \leq i \leq n\}$  with an edge from $(i,j)$ to $(k,l)$ if  the two cells are adjacent in the same row or column and  $C_{ij}\subset C_{kl}$. 

First assume  $K = \O(n)$ so that $D = \hat D(y)$. If $(i,j) \in D$, then $\r_y(i-1,j) = \r_y(i,j-1) = \r_y(i,j)$ so $C_{ij}$ is contained in $C_{i-1,j}$ and $C_{i,j-1}$. If $(i,j) \notin D$, then either $j \geq y(i)$ or $i \geq y(j)$. In the first case, $\r_y(i-1,j) = \r_y(i,j)-1$, so $C_{i-1,j} \subset C_{ij}$; in the second case, $\r_y(i,j-1) = \r_y(i,j)-1$, so $C_{i,j-1}\subset C_{ij}$. Thus, each cell in $D$ is the source of   edges in the implication graph going both north and west,  while each cell not in $D$ is the target of an edge going either south or east.
It follows that every $(i,j) \in D$ can be reached by a directed path in the implication graph  starting at some cell in $\Ess(D)$,
and every $(i,j) \notin D$
can be reached by a directed path  
 starting at  either $(1,1)$ or a cell in $D$. Since $(1,1) \notin D$ only if $w(1) = 1$, in which case  $C_{1,1} = \Fl(n)$, we  conclude that each $C_{ij}$  contains  $C_{pq}$ for some $(p,q) \in \Ess(D)$, so $Y^K_y$ is the intersection of the sets $C_{pq}$ for $(p,q) \in \Ess(D)$, as desired.

Now  suppose $K = \Sp(n)$ so that $D = \Dfpf(y)$.
Since in this case $\beta$ is skew-symmetric and $y$ is fixed-point-free, the numbers  $\rank \(\beta|_{F_i \times F_i}\)$ and $\r_y(i,i)$ are always even and bounded above by $i$. Hence $C_{1,1} = \Fl(n)$. We claim, moreover, that $C_{i,i-1} \subset C_{ii}$ for all $1 <i \leq n$. To show this, suppose $F_\bullet \in C_{i,i-1}$. If $\rank \(\beta|_{F_i \times F_{i-1}}\) = \rank \(\beta|_{F_i \times F_i}\)$ or $\r_y(i,i-1) < \r_y(i,i)$ then clearly $F_\bullet \in C_{ii}$; otherwise, it must happen that $\rank \(\beta|_{F_i \times F_{i-1}}\) = \rank \(\beta|_{F_i \times F_i}\)-1$ is odd and $\r_y(i,i-1) = \r_y(i,i)$ is even,   so  the strict inequality $\rank \(\beta|_{F_i \times F_{i-1}}\)  < \r_y(i,i-1)$ holds, which again implies   $F_\bullet \in C_{ii}$.
From the claim just shown and \eqref{Y-eq}, we  deduce that $Y^K_y$ is the intersection of  $C_{ij}$ for   $1\leq j < i \leq n$. The proposition  thus follows as in the orthogonal case, by considering the implication graph  on the set of cells $\{(i,j) : 1 \leq j < i \leq n\}$ strictly below the diagonal.
\end{proof}
 
As for Schubert varieties, to each orbit closure $Y^K_y$ there is an associated cohomology class
 \be\label{coclass-eq} [Y^K_y] \in H^*(\Fl(n),\ZZ) \qquad\text{for }y \in \I_K.\ee
 The Borel isomorphism \eqref{borel-eq} 
identifies 
 these cohomology classes with elements of the quotient  $\cP_n / (\Lambda_n^+)$.
One has the following analogue of \eqref{BGG} describing the action of the divided difference operators on $ [Y^K_y]$.

\begin{theorem}
\label{WY1-thm}
Let $i \in [n-1]$ and $y \in \I_K$ and set $s=s_i$. Then
\[
\partial_i [Y^K_y] = \begin{cases}
[Y^K_{y \rtimes s}] & \text{if $y \rtimes s = s y s < y$}\\
2[Y^K_{y \rtimes s}] & \text{if $y \rtimes s  = y s < y$ and $y s \in \I_K$}\\
0 & \text{if $s \notin \DesR(y)$ or $y\rtimes s \notin \I_K$.}
\end{cases}
\]
\end{theorem}

\begin{proof}
Fix $i \in [n]$. Suppose $X$ is a subvariety of $\Fl(n)$, and consider the set (variety, in fact) $X'$ obtained by replacing each flag $F_{\bullet} \in X$ by all flags $E_{\bullet}$ with $E_j = F_j$ for $j \neq i$. Informally, we obtain $X'$ by removing any conditions defining $X$ that restrict $F_i$. If $\dim X' = \dim X + 1$, then there will be an integer $d$ such that each flag $E_{\bullet}$ in a dense open subset of $X'$ arises from $d$ distinct flags in $X$, in which case $\partial_i [X] = d[X']$. If, alternatively, $\dim X' \neq \dim X +1$, then  $\partial_i [X] = 0$. For a more detailed justification of these assertions, see \cite[Chapter 10]{Fulton1997}.

One can use the rank conditions on $Y_y^K$ to  understand  $(Y_y^{K})'$ and the integers $d$. For the cases when $y \rtimes s_i \in \I_K$, see \cite[Propositions 2.1 and 2.7]{Wyser} and \cite[\S1.3]{WY}. Suppose $y \rtimes s_i \notin \I_K$, which occurs only if $K = \Sp(n)$ and $y\rtimes s_i = ys_i$ so that $ys_i$ is not fixed-point-free. Then $y(i) = i+1$.  We claim that row $i$ and column $i$ of $\Ess(\Dfpf(y))$ are both empty. Column $i$ of $\Dfpf(y)$ is  empty by definition. Row $i$ of $\Ess(\Dfpf(y))$ is empty since if  $(i,k) \in \Dfpf(y)$ for some $k<i$, then we also have $(i+1,k) \in \Dfpf(y)$ (for otherwise $y(k) = i+1$), in which case $(i,k) \notin \Ess(\Dfpf(y))$. Thus, any of the rank conditions defining $Y^K_y$ in Proposition \ref{prop:essential-set} that involve $F_i$ are implied by others that do not, so $(Y^{K}_y)' = Y_y^{K}$ and therefore $\partial_i [Y_y^{K}] = 0$.
\end{proof}

The theorem shows that one may compute polynomial representatives for the cohomology classes \eqref{coclass-eq} just as for Schubert classes, i.e., by applying divided difference operators to suitable representatives 
 for the longest element $w_n \in S_n$.
 Wyser and Yong identify such representatives in \cite{WY}. 
To state their result,  recall the definitions of $\Upsilon^{\O(n)}_{w_n}$ and $\Upsilon^{\Sp(n)}_{w_{2n}}$ from \eqref{eq:wyser.yong.prod}.

\begin{theorem}[{Wyser and Yong \cite[Theorem 1.1]{WY}}]
\label{WY2-thm}
Let   $y \in \I_K$ and $u,v \in \cA(y,w_n)$. Then
\[ \Upsilon^K_{w_n} \equiv [Y^K_{w_n}]\qquand \partial_u \Upsilon^K_{w_n} =  \partial_v \Upsilon^K_{w_n}.\]
\end{theorem}

Let $y \in \I_K$ and   $u \in \cA(y,w_n)$. If $K=\O(n)$ then define 
$\Upsilon^{\O(n)}_y \omdef= 2^{\kappa(y)- \kappa(w_n)} \partial_u \Upsilon^{\O(n)}_{w_n}
$
and if $K=\Sp(n)$ then define 
$\Upsilon^{\Sp(n)}_y \omdef= \partial_u \Upsilon^{\Sp(n)}_{w_{n}}.$
 Theorems \ref{WY1-thm} and \ref{WY2-thm} show
that 
$\Upsilon^{K}_{y}
$
 is then a representative of $ [Y^{K}_y]$.
 These polynomials do not depend on the choice of $u$, and so are unambiguously indexed by $\I_K\subset S_n$.
 Wyser and Yong note in \cite[Theorem 1.1]{WY} that these representatives are nonnegative integer linear combinations of ordinary Schubert polynomials. We can identify this decomposition explicitly by showing that Wyser and Yong's polynomials are actually scalar multiples of involution Schubert polynomials.
 
Suppose $y \in \I(S_n)$ and $z \in \Ifpf(S_{2n})$.
Let $z_\infty \in \tilde \I_\fpf$ be the permutation of $\PP$ defined from $z$ by \eqref{infprod},
and choose any  $u\in\cA(y,w_n)$ and $v\in\cA(z,w_{2n})$.
In the next proof, we let 
\[\Upsilon_y \omdef =2^{-\kappa(y)} \Upsilon^{\O(n)}_y=2^{-\lfloor n/2\rfloor}\partial_u \Upsilon^{\O(n)}_{w_n}
\qquand
\Upsilon^\fpf_{z_\infty} \omdef= \Upsilon^{\Sp(2n)}_z= \partial_v \Upsilon^{\Sp(2n)}_{w_{2n}}.
\]
These polynomials are well-defined, independent of our choice of $n$,
by \cite[Theorem 1.4]{WY}.

\begin{theorem} \label{WYatom-thm}
Let
$ y \in \I(S_n)$ and $z \in \Ifpf(S_{2n})$. Then
$
2^{\kappa(y)} \hat \fkS_y = \Upsilon^{\O(n)}_{y}
$
and
$
\Sfpf_z =  \Upsilon^{\Sp(2n)}_z.
$
\end{theorem}

\begin{proof} 
It suffices to
 argue that  $\{ \Upsilon_y\}_{y \in \I(S_\infty)} $ and $ \{ \Upsilon^\fpf_z \}_{z \in \tilde \I_{\fpf}}$
have the properties in Theorem \ref{invschubdef-thm} and Corollary \ref{invschubdeffpf-cor}
that uniquely characterize  $\{ \hat \fkS_y\}_{y \in \I(S_\infty)}$ and $\{ \Sfpf_z \}_{z \in \tilde\I_\fpf}$. 
For this, we first claim  for all $y \in \I(S_\infty)$ and $z \in \tilde \I_\fpf$ and  $i,j \in \PP$  that
\[ 
\partial_i \Upsilon_y = \begin{cases} \Upsilon_{y\rtimes s_i} & \text{if $s_i \in \DesR(y)$} \\ 0&\text{otherwise}\end{cases}
\]
and
\[
\partial_j \Upsilon^\fpf_z = \begin{cases} \Upsilon^\fpf_{z\rtimes s_j} & \text{if $s_j \in \DesR(z)$ and $z\rtimes s_j \in \tilde \I_\fpf$} \\ 0&\text{otherwise.}\end{cases}
\]
Choose $n \in \PP$ such that   $y \in \I(S_n)$ and   $z = w_\infty$ for some $w \in \Ifpf(S_{2n})$.
Then $\Upsilon_y \in \cP_{n-1}$ and $\Upsilon^\fpf_z \in \cP_{2n-1}$, so the claim holds automatically when $i \geq n$ and $j \geq 2n$ since both sides of the two equations are zero.
When $s_i \in \DesR(y)$ and $s_j \in \DesR(z)$ and $z \rtimes s_j \in \tilde \I_\fpf$,
 the desired identities follow directly from the definitions and Proposition~\ref{atomdes-prop}.
Finally, suppose $i \in [n-1]$ is not a descent of $y$ and $j \in [2n-1]$ is not a descent of $z$.
Then  
by the preceding case 
$ \Upsilon_y = \partial_i \Upsilon_{y\rtimes s_i} $ and $ \Upsilon^\fpf_z = \partial_j \Upsilon^\fpf_{z\rtimes s_j}$
since $s_i \in \DesR(y\rtimes s_i)$ and $s_j \in \DesR(z\rtimes s_j)$.
As the divided difference operators square to zero, it follows that $\partial_i \Upsilon_y = \partial_j \Upsilon^\fpf_z = 0$.

It remains  to show that $\Upsilon_{1} = \Upsilon^\fpf_{1_\fpf} = 1$ and that if $z\rtimes s_i \notin \tilde \I_\fpf$ then $\partial_i \Upsilon_z^\fpf = 0$.
We know that $2^{\kappa(y)} \Upsilon_y \equiv [Y^{\O(n)}_y]$ if $y \in \I(S_n)$ and $\Upsilon^\fpf_z \equiv [Y^{\Sp(2n)}_{w}]$ if $z=w_\infty$ where $w \in \Ifpf(S_{2n})$. The rank conditions in Proposition \ref{prop:essential-set} show that $Y^{\O(n)}_{1} = \Fl(n)$ and $Y^{\Sp(2n)}_{v_n} = \Fl(2n)$,   so these orbit closures correspond to the identity elements in their corresponding cohomology rings.  As the Borel isomorphism \eqref{borel-eq} is an isomorphism of rings and since $\Upsilon_y$ and $\Upsilon^\fpf_z$ are evidently homogeneous polynomials, it is immediate 
that $\Upsilon_{1} =\Upsilon_1^{\O(n)} = 1 $ and $ \Upsilon^\fpf_{1_\fpf} = \Upsilon_{\wfpf_n}^{\Sp(2n)}= 1$.

Finally 
 suppose $z \in \tilde \I_\fpf$ and $j \in [2n-1]$ are such that $z \rtimes s_j \notin \tilde \I_\fpf$. As before, we then have $z=w_\infty$ for some  $w \in \Ifpf(S_{2n})$, and
 necessarily $j \in \DesR(w)$ and $w\rtimes s_j = ws_j \notin \Ifpf(S_{2n})$.  Theorem \ref{WY1-thm}  implies $\partial_j [ Y^{\Sp(2n)}_{w}] = 0$, so as $\Upsilon^\fpf_z \equiv [Y^{\Sp(2n)}_{w}]$ it follows that $\partial_j\Upsilon^\fpf_{z}  = 0$.
We conclude by Theorem \ref{invschubdef-thm} and Corollary \ref{invschubdeffpf-cor}
that $\Upsilon_y = \hat \fkS_y$ and $\Upsilon^\fpf_z = \Sfpf_z$ for all $y \in \I(S_\infty)$ and $z \in \tilde\I_\fpf$, which is what we needed to show.
\end{proof}

\subsection{Product formulas}
\label{product-sect}

Theorem \ref{WYatom-thm} establishes an explicit product formula for  $\hat\fkS_{w_n}$ and $\Sfpf_{w_{2n}}$, and in this section we generalize that result to the following class of  involutions:

\begin{definition} An involution $y \in \I(S_\infty)$  is \emph{weakly dominant} if it has the form
\[y = (1,b_1)(2,b_2)\cdots(k,b_k)\]
for 
some 
$k \in \NN$ and some distinct integers $b_1,b_2,\dots,b_k $ with $b_i > k$ for all $i \in [k]$.
\end{definition}

If $y=(1,b_1)(2,b_2)\cdots(k,b_k) \in \cI(S_n)$ is weakly dominant then we define $r(y) \in S_{n-k}$ to be the permutation given in one-line notation by
\be\label{r-eq} r(y) \omdef= [b_1-k,b_2-k,\dots,b_k-k, c_1,c_2,\dots,c_{n-2k}]\ee
where $c_1 < c_2 <\dots <c_{n-2k} $ are the  elements of $\{1,2,\dots,n-k\} \setminus \{b_1-k , b_2-k,\dots,b_k-k\}$.

\begin{example}\label{r-ex}
We have $r\( (1,6)(2,5)(3,8)\) = [3,2,5,1,4]$ and 
$
r(\cG_k) = 1
$
and
$r(w_{k}) = w_{\lceil k/2\rceil}.$
 More generally,  if $y=u^{-1} \cG_k u$ for any $u \in S_k$, then $y$ is weakly dominant with $r(y) = u$.
\end{example}

It follows from Section \ref{diagram-sect} that the permutation $r(y)$ has these basic properties:

\begin{observation}\label{invdiagram-obs}
If $y \in \cI(S_n)$ is weakly dominant with $k=\kappa(y)$ distinct 2-cycles, then $r(y)$ belongs to $S_{n-k}$ with largest descent at most $k$, so $D(r(y)) \subset [k]\times [n-k]$.
\end{observation}

The diagram $\D(\cG_k)$ is the transpose of the shifted shape of $\delta_{k+1}$.
For any permutation $u$, define 
\be
\label{e-eq}
E_k(u) = \{ (j+k,i) : (i,j) \in D(u)\}
\ee
 to be the transpose of $D(u)$, shifted down by $k$ rows.

\begin{lemma}\label{invdiagram-lem} Suppose $y \in \I(S_\infty)$ is weakly dominant and $k = \kappa(y)$. Then 
\[ \D(y) = \D(\cG_k) \cup E_k(r(y))
\qquand
\Dfpf(y) = \Dfpf(\cG_k) \cup E_k(r(y)).
\]
\end{lemma}

\begin{example}
Let $y = (1,6)(2,5)(3,8) = [6,5,8,4,2,1,7,3]$  as in Example \ref{r-ex}. Then 
\[
\D(y) = \left\{ \barr{ccc} \circ &. & . \\ \circ  & \circ & .\\ \circ  & \circ & \circ \\ \circ  & \circ & \circ \\ \circ & . & . \\ . & . & . \\ . & . & \circ
\earr
\right\},
\qquad
\D(\cG_3) = \left\{ \barr{ccc} \circ &. & . \\ \circ  & \circ & .\\ \circ  & \circ & \circ 
\earr
\right\},
\qquad
D(r(y)) = \left\{ \barr{cccc} \circ &  \circ & . &. \\ \circ & . & . & . \\ \circ & .& . & \circ  
\earr
\right\},
\]
and the lemma's claim that $\D(y) = \D(\cG_3) \cup E_3(r(y))$ is evident.
\end{example}

\begin{proof}[Proof of Lemma \ref{invdiagram-lem}]
By definition,   $y=(1,b_1)(2,b_2)\cdots(k,b_k)$  for distinct integers $b_i >k$.
It is clear that $\D(\cG_k)  \subset \D(y)$.
Fix $(i,j) \in \PP\times \PP$. We claim that 
$(j+k,i) \in D(y)$ if and only if $(i,j) \in D(r(y))$.
We have  $(j+k,i) \in D(y)$ if and only if 
\be\label{event}  j+k < y(i) \qquand i < y(j+k).\ee
The claim holds when $i \in [k]$ since then
 the first of these conditions  is  equivalent to $j < r(y)(i)$, while the second  may be rewritten as
\[ i < y(j+k) \quad \Leftrightarrow\quad j+k \notin \{ b_1,b_2,\dots,b_i\}
\quad \Leftrightarrow\quad j \notin r(y)([i])
\quad \Leftrightarrow\quad
i < r(y)^{-1}(j).\]
Suppose instead that  $i>k$. Then $(i,j) \notin D(r(y))$ by Observation \ref{invdiagram-obs}.
We can only have $i<y(j+k)$ if $j+k \notin \{b_1,b_2,\dots,b_k\}$,
but if this holds then $i < y(j+k) = j+k$, in which case we cannot have $j+k < y(i)$.
Thus,
 the conditions \eqref{event} 
never simultaneously hold
so  $(j+k,i) \notin D(y)$.  We conclude that the set of positions in $D(y)$ below the $k^{\mathrm{th}}$ row is precisely 
 $E_k(r(y))$. The latter set  is contained entirely below the diagonal since $D(r(y)) \subset [k]\times[n-k]$, so  the lemma follows.
\end{proof}

The following proposition shows that every dominant involution is weakly dominant.

\begin{proposition}\label{dom-prop}
Let $y \in \I(S_\infty)$. The following are equivalent:
\ben
\item[(a)] $y$ is dominant (i.e., 132-avoiding).
\item[(b)] $\hat D(y)$ is the transpose of a shifted shape.
\item[(c)] $y$ is weakly dominant and $r(y)$ is dominant.
\een
\end{proposition}

\begin{proof}
The equivalence of (a) and (b) is immediate from the definitions of  dominant permutations and involution Rothe diagrams.
By Lemma \ref{invdiagram-lem}, it is clear that if $y$ is weakly dominant then the transpose of $\hat D(y)$ is a shifted shape if and only if $D(r(y))$ is the diagram of a partition, i.e., $r(y)$ is dominant. It is a straightforward exercise to check that a 132-avoiding (i.e., dominant) involution is  weakly dominant, so it follows that (a) and (c) are equivalent.
\end{proof}

We now prove Theorem~\ref{t:dominvSchub}, which is the involution analogue of Proposition \ref{schubertfactor-prop}(a). 

\begin{theorem}\label{invSchuprod-thm}
Suppose $y \in \I(S_\infty)$ and $z \in \Ifpf(S_{\infty})$ are dominant. Then
\[
\hat\fkS_y = 2^{-\kappa(y)} \prod_{(i,j) \in \D(y)} (x_i+x_j)
\qquand
\Sfpf_z = \prod_{(i,j) \in \Dfpf(z)} (x_i+x_j).
\]
\end{theorem}

\begin{proof}

We prove  the first identity, since the formula for $\Sfpf_z$ follows by essentially the same argument.
To begin, it is helpful to note (see \cite[\S2.1]{Manivel}) that the Rothe diagram of a permutation $w$ is the complement in $\PP\times \PP$ of the {hooks}  through the points $(i,w(i))$ for $i \in \PP$, where the \emph{hook} through a cell $(i,j)$ is the set of positions of the form $(i+t,j)$ or $(i,j+t)$ for $t \in \NN$.
Assume $w$ is dominant, so that $D(w)$ is a partition.
Then the northwest corners of the complement of $D(w)$ are all of the form $(i,w(i))$.
Moreover, if $(i,w(i))$ is such a corner then $i$ is a descent of $w$ if and only if the $(i+1)^{\mathrm{th}}$ row of $D(w)$ is shorter than the $i^{\mathrm{th}}$ row.

Let $y \in \I(S_n)$ be a dominant involution.
The desired formula holds when $y = w_n$ by Theorem \ref{WYatom-thm}, so assume $y<w_n$ and that the product formula is valid for all dominant involutions $z \in \I(S_n)$ with $\ell(y)<\ell(z)$. The Rothe diagram of $D(y)$ is strictly contained in $D(w_{n}) = (n-1,n-2,\dots,2,1)$,
so we may define $j \in [n]$ to be minimal such that $(n-j,j) \notin D(y)$, and then define $i \in [n-j]$ to be minimal such that $(i,j) \notin D(y)$. The cell $(i,j)$ is then a northwest corner of the complement of $D(y)$, so  $j = y(i)$.
Moreover,
rows $i,i+1,\dots,n-j+1$ of $D(y)$ all have length $j-1$, so we must have $s_i \notin\DesR(y)$. Finally, we must also have $j \leq i$ since $D(y)$ is symmetric under transpose.

Using these facts and the interpretation of the  Rothe diagram as the complement of the hooks through the points of a permutation, it is a straightforward exercise to check that 
\be\label{ijadd}
 \D(y \rtimes s_i) = \D(y) \cup \{(i,j)\}.
\ee We omit the details, since the argument is easier to visualize  than to transcribe and is similar to the proof of \cite[Proposition 2.6.7]{Manivel}. 
By  Proposition \ref{dom-prop},
the identity \eqref{ijadd} 
 implies that $y\rtimes s_i$ 
is itself a dominant involution of greater length than $y$,
so by induction and Theorem \ref{invschubdef-thm} we obtain
\be\label{prodform-eq}
\hat{\fkS}_y = \partial_i \hat{\fkS}_{y \rtimes s_i} = \partial_i \left[2^{-\kappa(y\rtimes s_i)}   (x_i+x_j)   \prod_{(k,l) \in \D(y)} (x_k + x_l) \right].
\ee
The transposed shifted shape $\D(y)$ has the same number of cells in rows $i$ and $i+1$ and no cells in columns $i$ and $i+1$.
The product $\prod_{(k,l) \in \D(y)} (x_k + x_l)$ is therefore $s_i$-invariant since the map  $(k,l) \mapsto (s_i(k),s_i(l))$ preserves $\D(y)$.
Since $\partial_i (fg) = f \partial_i(g)$ when $s_i f=f$, 
equation \eqref{prodform-eq}
becomes
\be\label{becomes-eq} \hat \fkS_y = 2^{-\kappa(y\rtimes s_i)}  \prod_{(k,l) \in \D(y)}  (x_k + x_l) \cdot \partial_i  (x_i+x_j)  .\ee
It is easy to check that $\partial_i( x_i + x_j ) $ is 1 if $i\neq j$ and 2 otherwise.
In turn, since $\kappa(\sigma)$ is the number of diagonal cells in $\D(\sigma)$ for any $\sigma \in \I(S_\infty)$, it follows that
 $\kappa(y\rtimes s_i) - \kappa(y)$ is 0 if $i\neq j$ and 1 otherwise.
Applying these observations transforms \eqref{becomes-eq} to the the desired formula for $\hat \fkS_y$.
\end{proof}

 For $p,q \in \NN$ define $\Phi_{p,q}$ to be the map $\cP_{\infty}(x; y) \to \cP_{p+q}$ with
 \be\label{Phi-eq}
 \Phi_{p,q} : f(x;y) \mapsto 
f(x_1,x_2, \dots, x_p, 0, 0, \dots; -x_{p+1}, -x_{p+2}, \dots, -x_{p+q}, 0, 0, \dots).
\ee
In other words, $\Phi_{p,q}$ is the ring homomorphism which maps $x_i \mapsto x_i$  for $i \in [p]$ and $y_j \mapsto -x_{p+j}$ for $j \in [q]$, while mapping all other variables to zero.
Suppose $z \in \cI(S_n)$ is weakly dominant and $k=\kappa(z)$. If $E_k(u)$ is defined as before Lemma \ref{invdiagram-lem}, then it follows from Observation \ref{invdiagram-obs} 
that
\be\label{phifact-eq} \prod_{(i,j) \in E_k(r(z))} (x_i+x_j)  = \prod_{(i,j) \in D(r(z))} (x_i + x_{j+k}) = \Phi_{k,n-k}\( \prod_{(i,j) \in D(r(z))} (x_i-y_j)\).\ee
This fact leads to the following result,  generalizing the previous theorem.

\begin{theorem}\label{factor-thm} 
Suppose $z \in \I(S_\infty)$ is a weakly dominant involution. Let $p = \kappa(z)$, define $n$ to be the smallest integer such that $z \in S_n$, and set $q= n-p$.
Then
\[
\hat\fkS_z = \hat\fkS_{\cG_{p}} \cdot \Phi_{p,q}\(\fkS_{r(z)}(x;y)\)
\qquand
\Sfpf_z = \Sfpf_{\cG_{p}} \cdot \Phi_{p,q}\(\fkS_{r(z)}(x;y)\)
\]
where the second identity applies only in the case when $z$ is fixed-point-free.
\end{theorem}

\begin{proof}
Since $z$ is weakly dominant we may write $z = (1,b_1)(2,b_2)\cdots(p,b_p)$ for distinct  integers $b_i$ all greater than $p$.
First suppose $b_1>b_2>\dots>b_p$. One checks that $z$ is then 
 dominant,  so it follows from Proposition \ref{schubertfactor-prop}(a) that the right-most expression in \eqref{phifact-eq} is  precisely $\Phi_{p,q}\(\fkS_{r(w)}(x;y)\)$. 
By Proposition \ref{grassman-prop}, noting Example \ref{invRothe-ex}, it always holds that $\hat \fkS_{\cG_p} = \prod_{(i,j) \in \hat D(\cG_p)} (x_i+x_j)$,
 so the desired formula for $ \hat\fkS_w $ follows by  Lemma \ref{invdiagram-lem} and Theorem \ref{invSchuprod-thm}.

Suppose alternatively that there exists an index $i \in [p-1]$ such that $b_i < b_{i+1}$. Then $i$ is a (right) ascent of both $z$ and $r(z)$, and evidently $z\rtimes s_i = s_i zs_i $ is also weakly dominant, so we may  assume by induction that 
$ \hat\fkS_{z\rtimes s_i} = \hat\fkS_{\cG_{p}} \cdot \Phi_{p,q}\(\fkS_{r(z\rtimes s_i)}(x;y)\) $.
As $\hat \fkS_{\cG_p} \in \Lambda_p$ by Proposition \ref{grassman-prop}, it follows
by our inductive hypothesis and Theorem \ref{invschubdef-thm}  that
\[ \hat\fkS_z = \partial_i \hat \fkS_{z\rtimes s_i}  = \partial_i \left[ \hat\fkS_{\cG_{p}} \cdot \Phi_{p,q}\(\fkS_{r(z\rtimes s_i})(x;y)\) \right] = \hat \fkS_{\cG_p} \cdot \partial_i \Phi_{p,q} \(\fkS_{r(z\rtimes s_i)}(x;y)\).  \]
Since 
   $r(z\rtimes s_i) = r(z)s_i > r(w)$, and since $\partial_i$ acts  only on the $x_i$ variables when applied to an element of $\cP_\infty(x;y)$,
we have  
$\partial_i \Phi_{p,q} \(\fkS_{r(z\rtimes s_i)}(x;y) \)
=  \Phi_{p,q}\(\fkS_{r(z)}(x;y)\).$
 Substituting this into the preceding equation gives the desired formula for $\hat \fkS_z$.
When $z$ is fixed-point-free, the  analogous identity for $\Sfpf_z$ 
 follows by a   similar argument.
\end{proof}


\subsection{Involution Stanley symmetric functions}
\label{invstan-sect}

The 
 \emph{involution Stanley symmetric function} indexed by $y,z \in \I(S_\infty)$ is
$ 
\hat F_{y,z} = \sum_{u \in \cA(y,z)} F_u \in \Lambda.
$
We abbreviate as usual by setting
$
\hat F_y = \hat F_{1,y}$ and 
$
 \Ffpf_z = \hat F_{\wfpf_n,z}$ for $z \in \Ifpf(S_{2n})$. 
Observe that $\hat F_{y,z} = 0$ if $y \not \leq_T z$, with $<_T$ as in Section \ref{gen-sect}.
The following slight modification to Theorem \ref{stanley-def} holds by Theorem-Definition \ref{atoms-thmdef}: 

\begin{observation} If $y,z \in \I(S_\infty)$ then $\hat F_{y,z} = \sum_{\lambda} \beta_{y,z,\lambda} s_\lambda$ where 
the sum is over  partitions $\lambda$  and $\beta_{y,z,\lambda}$ is the number of strict tableaux $T$ of shape $\lambda$ with $\rww(T)\in \hat\cR(y,z)$. 
\end{observation}

We have $\beta_{y,z,\lambda} = 0$ if $\lambda $ is not a partition of $\ellhat(y,z)$, so the sum   appearing in the observation's formula for $\hat F_{y,z}$ is finite. 
From
Theorem \ref{EG-cor}, we obtain the following corollary:

\begin{corollary}\label{tildeF-cor} If $y,z \in \I(S_\infty)$  then
 $|\hat\cR(y,z)| = \sum_{\lambda} \beta_{y,z,\lambda} f^\lambda$.
\end{corollary}
Thus, to count the number of elements in the sets $\hat \cR(y,z)$, we need only determine the Schur decomposition of the symmetric functions $\hat F_{y,z}$. 
The rest of this section is spent proving a few facts about such decompositions that follow directly from properties of Stanley symmetric functions and involution words.

Just as for ordinary Stanley symmetric functions, $\hat F_y$ and $\Ffpf_y$ are skew Schur functions when indexed by  321-avoiding permutations.
In detail, given a sequence of nonnegative integers $c=(c_1,c_2,\dots,c_n)$ whose nonzero entries occur in positions $k_1<\dots<k_l$, let $\skew(c)$ be the set of cells $(i,j) \in \PP\times \PP$, with $1\leq i \leq l$,
such that 
\[ i-k_i - c_{k_i} < j +(l -k_l-c_{k_l}) \leq i-k_i.\]
By  \cite[\S2.2.2]{Manivel}, it follows that if $c=c(w)$ is the code of a 321-avoiding permutation $w$, then $\skew(c)$ is a skew shape.
For example, 
if $y =\wfpf_n\in S_{2n}$ then $c(y) = (1,0,1,0,\dots,1,0)$ and $\skew(c(y)) = \delta_{n+1}/\delta_n$
where
$\delta_n = (n-1,n-2,\dots,2,1)$.

If $y \in S_\infty$ is 321-avoiding then Proposition \ref{Fskew-prop}  
asserts that $F_y = s_{\skew(c)}$. The following parallel statement holds for $\hat F_y$ and $\Ffpf_y$.
The codes $\hat c(y)$ and $\hat c_\fpf(y)$ used here
are defined by \eqref{inv-code-eq}.

\begin{proposition}\label{hatFskew-prop}
Suppose $y \in \I(S_\infty)$ is 321-avoiding.
Then 
 $ \skew(\hat c(y)) = \lambda/\mu $ is a skew shape and $ \hat F_y = s_{\lambda/\mu}$.
 If $y$ is fixed-point-free, then $\skew(\hat c_\fpf(y)) = \gamma/\nu$ is a skew shape and $\Ffpf_y = s_{\gamma/\nu}$.
\end{proposition}

\begin{proof}
By Theorem \ref{atomicSn-thm} and Proposition \ref{lexatom-prop},  $\cA(y) = \{ \alpha_{\min}(y) \}$ and (in the fixed-point-free case) $\cAfpf(y) = \{\beta_{\min}(y)\}$. By Corollary \ref{lexatom-cor}, $\alpha_{\min}(y) $ and $\beta_{\min}(y)$ are themselves  321-avoiding, so the result  follows from its analogue for ordinary Stanley symmetric functions by Lemma \ref{invcode-lem}.
\end{proof}

We next classify  the involutions $y \in S_\infty$ for which  $\hat F_y$ and $\Ffpf_y$ are   Schur functions.
To begin, we note the following lemma 
which derives from the discussion after \cite[Proposition 5.4]{EL}.

\begin{lemma}[Eriksson and Linusson \cite{EL}] \label{EL-lem} A permutation $w \in S_\infty$ is both 321-avoiding and 2143-avoiding if and only if either $w$ or $w^{-1}$ is Grassmannian.
\end{lemma}

If $w \in S_k$ and  $m \in \NN$ then we define the shifted permutation
\be\label{1cross-def} 1_m \times w \omdef= [1,2,\dots,m,w(1)+m,w(2)+m,\dots w(k)+m ] \in S_{k+m}.\ee
For any  $v \in S_m$  we similarly define $v \times w = v \cdot (1_m \times w) \in S_{k+m}$. 
With this convention, if $ z \in S_{2k}$ is a fixed-point-free involution, then $\shiftfpf{m}{n}{z}$ is as well.
The operation on permutations just defined, which we denote with $\times$,
is sometimes called the ``direct sum'' and denoted with $\oplus$.

\begin{example}
We have 
$\shift{4}{w_4} = (5,8)(6,7)
$
and
$ \shiftfpf{2}{2}{w_4} = (1,2)(3,4)(5,8)(6,7)(9,10).
$
\end{example}

Recall from \eqref{g-def} that $\cG_k = (1,k+1)(2,k+2),\dots(k,2k) \in \I(S_{2k})$.

\begin{proposition}\label{grass-prop} If $y \in S_\infty$ is a Grassmannian involution then $y= \shift{m}{\cG_{k}}$ for some $m,k \in \NN$.
\end{proposition}

\begin{proof}
Let $y \in \I(S_\infty)$ be Grassmannian. 
It suffices to show that if $k=y(1)-1$ is positive then $y = \cG_k$.
For this, observe that the Rothe diagram $D(y)$ contains the cells $(1,i)$ for all $i \in [k]$, and therefore also the cells $(i,1)$ for $i \in [k]$ since $D(y) = D(y^{-1}) = D(y)^T$.
As $k$ is evidently the unique descent of $y$, $D(y)$ has no cells below the $k^{\mathrm{th}}$ row or (by symmetry) to the right of the $k^{\mathrm{th}}$ column; hence $D(y) \subset [k]\times [k]$. 
On the other hand, by the definition of a Grassmannian permutation preceding Proposition \ref{schubertfactor-prop},
it holds that each nonempty row in $D(y)$ contains at least as many cells as the row above it. Since the first row of $D(y)$ already has $k$ cells, it follows  that in fact $D(y) = [k]\times [k] = D(\cG_k)$, so $y= \cG_k$ as desired.
\end{proof}

The permutations $y \in \I(S_\infty)$ whose involution Stanley symmetric functions are single Schur functions (of straight shape) turn out to have a very restricted form, which we now describe.

\begin{proposition}\label{Fgrass-prop}
Suppose $y = \shift{m}{\cG_{k-1}}$ and $z = \shiftfpf{m}{n}{\cG_{k}}$ for some $m,n,k$. Then 
\[\hat F_y = \Ffpf_z = s_{\delta_k}
\qquand
|\hat\cR(y)| = |\cRfpf(z)| = f^{\delta_k}.\]
\end{proposition}

\begin{proof}
In view of Example~\ref{cg-ex}(a),
this follows from Corollary \ref{tildeF-cor} and Proposition \ref{hatFskew-prop}.
\end{proof}

\begin{theorem}\label{grass-thm}
Let $y \in \I(S_\infty)$ and $z \in \Ifpf(S_{\infty})$.
\ben
\item[(a)] $\hat F_y$ is a Schur function if and only if $y = \shift{m}{\cG_k}$ for some $m,k \in \NN$.
\item[(b)] $\Ffpf_z$ is a Schur function if and only if $z = \shiftfpf{m}{n}{ \cG_k}$ for some $m,n \in \NN$ and $k \in \PP$.
\een
\end{theorem}

\begin{proof}
Suppose $\hat F_y$ is a Schur function.
Combining Theorems  \ref{Mac-thm} and \ref{atomicSn-thm} with Proposition \ref{lexatom-prop} and its corollary shows that  $y$ must be 321-avoiding
and that $\alpha_{\min}(y)$ must be both 321-avoiding and 2143-avoiding. 
By Lemma \ref{EL-lem}, either $\alpha_{\min}(y)$
or its inverse is therefore Grassmannian. It is apparent from the definition that $\alpha_{\min}(y)^{-1}$ is Grassmannian only if $y$ is the identity or a simple transposition
(as the only 321-avoiding transpositions are adjacent transpositions),
  in which case $\alpha_{\min}(y) $ is equal to its inverse. We conclude that $\alpha_{\min}(y)$ must be Grassmannian. 
  Since $\cA(y) = \{\alpha_{\min}(y)\}$, both $y$ and $\alpha_{\min}(y)$ have the same right descent set.
  Therefore $y$ is also Grassmannian, so  $y = \shift{m}{\cG_k}$ for some $m,k \in \NN$ by Proposition~\ref{grass-prop}.

Suppose next that $\Ffpf_z$ is a Schur function. By the same set of results as  cited in the previous paragraph, it   follows that $z$ must be 321-avoiding and either $\beta_{\min}(z)$ or its inverse must be Grassmannian.
For $i=1,2,3,4$, let $(a_i,b_i)$  be elements of the set $\{ (a,b)  \in \PP\times \PP : a < b = z(a)\}$.  One checks that it cannot occur that $a_1<a_2 <b_2<b_1$ (as then $z$ would contain the pattern 321) or that $a_1<a_2<b_1<a_3<b_2<b_3$ (as then $\beta_{\min}(z)$ and its inverse would both  contain the pattern 132546) or that $a_1<a_2<b_1<b_2 < a_3<a_4<b_3<b_4$ (as then $\beta_{\min}(z)$ and its inverse would both contain the pattern 13245768).
Using these properties, it is an  elementary exercise to deduce that $z$ must have the form $\shiftfpf{m}{n}{\cG_k}$ for some $m,n,k \in \NN$. 

This proves one half of the theorem, and the converse holds by Proposition \ref{Fgrass-prop}.
\end{proof}

A permutation is \emph{antivexillary} if it is both 321-avoiding and 351624-avoiding. For an explanation of this terminology, see \cite[Proposition 5.1]{EL}. 

\begin{corollary} Let $y \in \I(S_\infty)$ and $z \in \Ifpf(S_{2n})$. 
\ben
\item[(a)] $\hat F_y$ is a Schur function if and only if $y$ is  vexillary and 321-avoiding.
\item[(b)] $\Ffpf_z$ is a Schur function if and only if $z$ is antivexillary and 231564-avoiding.
\een
\end{corollary}

\begin{proof}
Lemma \ref{EL-lem} 
implies that $y$ is  vexillary and 321-avoiding
if and only if $y$ is Grassmannian,
which occurs if and only if $y=1_m\times \cG_k$ for some $m,k \in \NN$ by Proposition \ref{grass-prop}.
Part (a) follows by combining this observation with Theorem~\ref{grass-thm}(a).

To prove part (b), one must show that a fixed-point-free involution $z$ is antivexillary and 231564-avoiding
if and only if $z = \shiftfpf{m}{n}{\cG_k}$ for some $m,k \in \NN$. 
From the second paragraph in the proof of Theorem~\ref{grass-thm}, we rule out three cases that correspond to containing the patterns $321$, $351624$ and $34127856$ respectively.
The result follows by noting $34127856$ contains $231564$ and observing that any $z \in \Ifpf(S_{2n})$ containing $231564$ must also contain one of these three patterns.
\end{proof}

\subsection{Stabilization}
\label{stabilization-sect}

To prove stronger statements about involution Stanley symmetric functions, we must leverage the results in Section \ref{product-sect}; we discuss methods for this here.
If $f$ is a power series in the variables $x=\{x_1,x_2,\dots\}$, then we write $r_n(f)$ or $f(x_1,\dots,x_n)$ for the power series formed by setting 
 the variables $x_{n+1}= x_{n+2}= \dots = 0$.
If $f\in \cP_\infty$ then $r_n(f) \in \cP_n$
and if $f \in \Lambda$ then $r_n(f) \in \Lambda_n$, where  $\Lambda_n$ denotes the ring of polynomials in $\cP_n$ fixed by the  action of $S_n$.

Let $w \in S_\infty$.
Following \cite{Macdonald}, we define
the  \emph{stabilization (in degree $n \in \NN$)} of $\fkS_w$ to be
\be\label{stab-obs} \stb_n(\fkS_w) \omdef= F_w(x_1,\dots,x_n) \in \Lambda_n.\ee
By Proposition \ref{schubbasis-prop},   this formula extends by linearity to a map $\stb_n:\cP_\infty \to \Lambda_n$. We then have
\[ \stb_n( \hat \fkS_{y,z} ) = \hat F_{y,z}(x_1,\dots,x_n)\qquad\text{for $y,z \in \I(S_\infty)$}
\]
and  $\lim_{n\to \infty} \stb_n(\hat \fkS_{y,z}) = \hat F_{y,z}$, where  the limit is interpreted in the sense of formal power series, with a sequence of  power series defined to be convergent if
the sequence of coefficients of any fixed monomial is eventually constant.

By applying   $\stb_n$  to both sides of the identities in Theorem \ref{factor-thm},
one might hope to show that  similar
factorizations hold for   $\hat F_y$ and $\Ffpf_z$. This strategy cannot work in general, since stabilization  is not a ring homomorphism and  may fail to preserve products.
However, we will find that in certain cases of interest the maps $\stb_n$ do behave as we would wish. To prove this we will require several preliminaries about these operations.
The following statement is immediate from \eqref{schub1-eq}.

\begin{proposition}
[{\cite[Proposition 2.8.1]{Manivel}}]
 \label{stab-cor}
If $w \in S_\infty$ then
$\stb_n(\fkS_w) = \fkS_{1_N \times w}(x_1,x_2,\dots,x_n)$ for all $N \geq n$.
\end{proposition}

More usefully, 
we can express $\stb_n$ in terms of certain modified divided differences.
Following \cite{Macdonald}, we define the \emph{isobaric divided difference operator} $\pi_i : \cP_\infty \to \cP_\infty$ for $i \in \PP$
by 
\be\label{pi-i-eq}
\pi_i f   \omdef= \partial_i (x_if).
\ee 
For example, $\pi_i(x_ix_{i+1}) = \partial_i(x_i^2x_{i+1})=x_ix_{i+1}$.
One checks that $\pi_i^2 = \pi_i$ and 
\be\label{pi-property}
\pi_i ( fg) = f \cdot \pi_i g\qquad\text{whenever $f,g \in \cP_\infty$ and $s_i f=f$.}
\ee
In particular if $s_if=f$ then $\pi_i (f) = f \cdot \pi_i(1) = f$.
These operators,
like the ordinary divided differences $\partial_i$, satisfy the Coxeter relations \eqref{coxrel}.
Therefore,
for $w \in S_\infty $  we may   define 
\be
\label{isobar-eq}
\pi_w \omdef= \pi_{i_1}\cdots \pi_{i_k}\qquad\text{for any reduced word $(s_{i_1},\dots,s_{i_k}) \in \cR(w)$.}
\ee
We will require the following property, which is less well-known.

\begin{lemma} \label{lem:restriction-and-isobaric}
If $i,n \in \PP$ and $f \in \cP_\infty$ then
$
 r_n( \pi_i f ) = \begin{cases}
 \pi_i r_n(f) & \text{if $i < n$}\\
 r_n(f) & \text{if $i \geq n$}.
 \end{cases}
$
\end{lemma}

\begin{proof}
Checking the lemma is a simple exercise in algebra which we leave to the reader.
\end{proof}

The next theorem  appears in \cite{Macdonald}, but since these notes are out of proof and difficult to obtain, we include a self-contained proof.

\begin{theorem}[{Macdonald \cite[Eq.\ (4.25)]{Macdonald}}] \label{stabpi-thm}
For all $f \in \cP_n$ it holds that $\stb_n (f) = \pi_{w_n}f$.
\end{theorem}

\begin{proof}
Define the operator $\tau_n = \pi_1 \cdots \pi_n$.
 Since $x_i( \pi_j f) = \pi_j (x_if)$ for $i < j$, it follows that we may also write $\tau_n f  = \partial_1 \cdots \partial_n (x_1 \cdots x_n f)$ for $f \in \cP_\infty$.
Suppose $u \in S_\infty$ has largest descent at most $n$. We  claim that  $\tau_n \fkS_u = \fkS_{1_1\times u}$
where 
$1_1\times u = [1,u(1)+1,u(2)+1,\dots] \in S_\infty$ as in \eqref{1cross-def}.
To show this, first assume $u \in S_n$ 
and let 
$v = w_{n+1} w_n u =  [u(1) +1, u(2)+1,\dots,u(n)+1,1] \in S_{n+1}$.
 Since the product $x_1\cdots x_n$ is invariant under $S_n$, it  holds by Theorem \ref{fkSw0-thm} that
\[ x_1\cdots x_n  \fkS_{ u} = x_1\cdots x_n \partial_{u^{-1}w_n} \fkS_{w_n} =  \partial_{u^{-1}w_n} (x_1\cdots x_n \fkS_{w_n}) = \partial_{u^{-1}w_n} \fkS_{w_{n+1}} = \fkS_{v}.\]
Therefore
$
\tau_{n} \fkS_u = \partial_1 \cdots \partial_n (x_1 \cdots x_n \fkS_{u})
= \partial_1 \cdots \partial_n  \fkS_{v}
$.  It is clear that we have a descending chain
\[v > v s_n > vs_n s_{n-1}> \dots > v s_ns_{n-1} \cdots s_1 = 1_1\times u.\] Thus, we conclude by Theorem \ref{schubunique-thmdef} that $\tau_n \fkS_u = \fkS_{1_1\times u}$ when $u \in S_n$.
To prove the claim in general, observe that $\fkS_u \in \cP_n$ by Proposition \ref{schubbasis-prop}, so $\pi_i \fkS_u = \fkS_u$ for all $i>n$.
Therefore if $u \in S_N$ for some $N\geq n$, then $\tau_N\fkS_u = \tau_n \fkS_u =  \fkS_{1_1\times u}$  by the part of the claim already shown.

Fix $f \in \cP_n$.
Given our claim, it follows by Proposition~\ref{schubbasis-prop}
and \ref{stab-cor} that  
\[\stb_n(f) = r_n \(\tau_{2n-1} \cdots \tau_{n+1} \tau_n f\).\]
One checks  using Lemma \ref{lem:restriction-and-isobaric} that if $N\geq n$ then   $r_n( \tau_N g) = \tau_{n-1} r_n(g)$ for all $ g \in \cP_\infty$. Using this property, we deduce that
\[ \stb_n(f) =r_n \(\tau_{2n-1} \cdots \tau_{n+1} \tau_n f\)= \tau_{n-1}^n r_n(f) = \tau_{n-1}^n f .\]
Since $\pi_i^2 = \pi_i$ for all $i \in \PP$, if $w \in W$ and $s \in S$, then  $\pi_w \pi_s$ is equal to $\pi_w$ when $s \in \DesR(w)$ and to $ \pi_{ws}$ when $s\notin \DesR(w)$. Using this property, it is a simple exercise to check that $\tau_{n-1}^n = \pi_{w_n}$, and this suffices to complete the proof.
\end{proof}

 We may now begin to say something about the ``stability'' of the formulas in Theorem \ref{factor-thm}.
 In view of the preceding theorem and \eqref{pi-property}, it follows that $\stb_n(fg) = f  \stb_n(g)$ if $f \in \cP_\infty$ is invariant under the action of $S_n$. We would like to apply something like this identity to Theorem \ref{factor-thm}, but, problematically, the involution Schubert polynomials $\hat \fkS_{\cG_k}$ and $\Sfpf_{\cG_k}$ appearing in that result are  symmetric only under the action of the subgroup $S_k \times S_{n-k}$, not all of $ S_n$. To get around this difficulty, we factor $\stb_n$ into two operators, one of which respects the partial symmetry which we encounter, in the following way.

Fix nonnegative integers $p,q$ with $n=p+q$, and write $\Lambda_{p\times q}$ for the subring of polynomials in $\cP_n$ which are fixed by the action of $S_p \times S_q \subset S_n$. Thus $\Lambda_n = \Lambda_{0\times n} = \Lambda_{n\times 0}\subset \Lambda_{p\times q}$.
Let 
\be\label{g-p-q-eq} 
\cG_{p,q} \omdef= w_n \cdot (w_p \times w_q) = [ q+1,q+2,\dots,n,1,2,\dots,q] \in S_n
\ee
and define $\stb_{p,q} : \cP_n \to \Lambda_{p\times q}$ and $\stb_{n/p,q} : \Lambda_{p\times q} \to \Lambda_n$ by 
\be\label{stb-p-q-eq}
 \stb_{p,q} \omdef= \pi_{w_p\times w_q} \qquand \stb_{n/p,q} \omdef= \pi_{\cG_{p,q}}.
 \ee
It is clear from Theorem \ref{stabpi-thm} that 
$ \stb_n(f) = \stb_{n/p,q}\stb_{p,q}(f)$ for all $f \in \cP_n.$

\begin{lemma}\label{stab-lem} Let $p,q \in \NN$ and $n=p+q$.
Suppose
$f \in \Lambda_{p\times q}$ and $g \in \cP_\infty$ is such that 
$\stb_{p,q}(g) \in \Lambda_{n}$. Then
$
\stb_n(fg) = \stb_{n}(f)  \stb_{p,q}(g).
$
\end{lemma}

\begin{proof}
Our hypotheses together with \eqref{pi-property} imply that $\stb_{p,q}(f)= f$ and 
\[
\stb_{p,q}(fg) = f\cdot \stb_{p,q}(g)
\qquand
\stb_{n/p,q}(f\cdot \stb_{p,q}(g)) = \stb_{n/p,q}(f) \stb_{p,q}(g)
.\]
Hence $ \stb_n(fg) = \stb_{n/p,q}  \stb_{p,q}(fg)= 
 \stb_{n/p,q} ( f)    \stb_{p,q}(g)
 =
  \stb_{n} ( f)    \stb_{p,q}(g)
.$
\end{proof}

The algebra of symmetric functions $\Lambda$ may be identified with its graded dual and so given the structure of a graded, self-dual Hopf algebra; see \cite[Chapter 2]{ReinerNotes} for the details of this standard construction.
The coproduct $\Delta : \Lambda \to \Lambda \otimes \Lambda$ of this Hopf algebra
is defined to be the linear map satisfying 
$
\Delta(f) = \sum_i g_i\otimes h_i
$ where $g_i, h_i \in \Lambda$ are symmetric functions such that
\[f(x_1,x_2,\dots,y_1,y_2,\dots) = \sum_i g_i(x_1,x_2,\dots)  h_i(y_1,y_2,\dots).
\]
For any partition $\nu$, it holds that $\Delta(s_\nu) = \sum_{\lambda,\mu} c^\nu_{\lambda,\mu} s_\lambda\otimes s_\mu$ where $c^\nu_{\lambda,\mu}$ are the Littlewood-Richardson coefficients.
Write $\omega : \Lambda \to \Lambda$ for the linear map defined by $\omega(s_\nu) = s_{\nu^T}$ for all partitions $\nu$, where $\nu^T$ denotes the transpose of $\nu$.

\begin{lemma}\label{lastthm-lem2}
Let $p,q \in \NN$ and $n=p+q$. Suppose $w \in S_q$ has largest descent at most $p$. If 
\[
(\id\otimes \omega)\circ \Delta(F_w) = \Delta(F_w) 
\]
then
$ \stb_{p,q}\(\Phi_{p,q}\(\fkS_w(x;y)\)\) = F_w(x_1,\dots,x_n)$ where $\Phi_{p,q}$ is the map defined by \eqref{Phi-eq}.

\end{lemma}

\begin{proof}
It is clear from Definition \ref{doubleschub-def} and \eqref{Phi-eq}
that
\[
 \Phi_{p,q}\(\fkS_{w}(x;y)\) 
 =
 \sum_{\substack{ w = v^{-1}u \\ \ell(w) = \ell(u)+\ell(v)}} \fkS_u(x_1,x_2,\dots,x_p,0,0,\dots) \fkS_v(x_{p+1},x_{p+2},\dots,x_n,0,0,\dots).
 \]
 In the right hand sum, all indices $u$   have largest descent at most $p$ (as this is the largest possible descent of $w$), while all indices $v$ have largest descent at most $q$ (as $ v\in S_{q}$).
By Proposition \ref{schubbasis-prop}, we can therefore drop the trailing zeros  and simply write
 \[
 \Phi_{p,q}\(\fkS_{w}(x;y)\)
=
 \sum_{\substack{ w = v^{-1}u \\ \ell(w) = \ell(u)+\ell(v)}} \fkS_u(x_1,x_2,\dots,x_p) \fkS_v(x_{p+1},x_{p+2},\dots,x_n) \in \cP_n.
\]
Since $\stb_{p,q}$    acts on $ \Phi_{p,q}\(\fkS_{w}(x;y)\)$ as the operator $\pi_{w_p \times w_{q}}$,
it follows via Theorem \ref{stabpi-thm} that
\be\label{stbphi-eq}
 \stb_{p,q}\( \Phi_{p,q}\(\fkS_{w}(x;y)\)\)
 =
  \sum_{\substack{ w = v^{-1}u \\ \ell(w) = \ell(u)+\ell(v)}} F_u(x_1,x_2.\dots,x_p) F_v(x_{p+1},x_{p+2},\dots,x_n).
\ee
Now, we have from \cite[Proposition 5 and Theorem 12]{Lam2} 
that $\Delta(F_w) = \sum_{} F_u \otimes F_v$ where the sum is over all $u,v \in S_q$ with $w=vu$ and $\ell(w) = \ell(v) + \ell(u)$.
On the other hand,  Macdonald \cite[Corollary 7.22]{Macdonald} proves that $\omega(F_w) = F_{w^{-1}}$. Hence, if $(\id\otimes \omega)\circ \Delta(F_w) = \Delta(F_w)$, then 
\[F_w(x_1,x_2,\dots,y_1,y_2,\dots) = \sum_{\substack{w=v^{-1}u \\ \ell(w) = \ell(v) + \ell(u)}} F_u(x_1,x_2,\dots)  F_v(y_1,y_2,\dots).
\]
On transposing the variables $y_i$ and $x_{p+i}$ for $i \in [q]$ (which by symmetry does not affect either expression) and then  setting $x_{n+i}= y_{i} = 0$ for $i \in \PP$, the left side of this identity becomes $F_w(x_1,\dots,x_n)$ while the right side becomes the formula \eqref{stbphi-eq} for 
$\stb_{p,q}\( \Phi_{p,q}\(\fkS_{w}(x;y)\)\)$;   these expressions are therefore equal when $\id \otimes \omega$ fixes $\Delta(F_w)$.
\end{proof}

For $n \in \NN$ let $p_n \omdef= x_1^n + x_2^n + \dots \in \Lambda$ denote the usual  \emph{power sum symmetric function}.
Since $\Delta(p_n) = p_n \otimes 1 + 1\otimes p_n$ \cite[Proposition 2.3.6(a)]{ReinerNotes} and $\omega(p_n) = (-1)^{n-1} p_n$ \cite[Proposition 2.4.1(a) and Eq.\ (2.4.7)]{ReinerNotes}, and since $\Delta$ and $\omega$ are algebra homomorphisms, 
it follows that $(1\otimes \omega)\circ \Delta(f) = \Delta(f)$ whenever $f $ belongs to the Hopf subalgebra 
\be\label{lambda-odd-eq}
\LambdaOdd \omdef=\QQ[p_1,p_3,p_5,\dots] \cap \Lambda
\ee
generated by the odd-indexed power sum symmetric functions.
This subalgebra   is studied in a few places (see, e.g., \cite{AS,HH,Stem_proj}), but does not seem to have an established name.
The following theorem is the main result of this section, and will imply the results described in the introduction.

\begin{theorem}\label{last-thm}
Let $y \in \I(S_\infty)$ be weakly dominant  with $k = \kappa(y)$. If  $F_{r(y)} \in\LambdaOdd$
then 
\[ \hat F_y = \hat F_{\cG_{k}} F_{r(y)} 
\qquand
\Ffpf_y = \Ffpf_{\cG_{k}} F_{r(y)}
\]
where the second identity applies only in the case when $y$ is fixed-point-free.
\end{theorem}

\begin{proof}
Let $n$ be the smallest integer such that $y \in S_n$, so that $n=2k$ when $y$ is fixed-point-free.
As noted in the preceding discussion,  the operator $\id \otimes \omega$ preserves $\Delta (F_{r(y)})$
and so 
applying Lemmas \ref{stab-lem}
 and \ref{lastthm-lem2} to Theorem \ref{factor-thm}
shows that
$
\hat F_y(x_1,\dots,x_n) = \hat F_{\cG_k}(x_1,\dots,x_n) F_{r(y)} (x_1,\dots,x_n)$
and, when $y$ is fixed-point-free, that  $\Ffpf_y(x_1,\dots,x_{n}) = \Ffpf_{\cG_k}(x_1,\dots,x_{n}) F_{r(y)} (x_1,\dots,x_{n})$.
It remains only to argue that these identities in $\Lambda_n$ lift to identities in $\Lambda$.

Write $\ell(\lambda)$ for the number of parts in a partition $\lambda$.
Let $\Lambda_{n,k}$ be the subspace of $\Lambda_n$ spanned by the Schur polynomials $s_\lambda(x_1,\dots,x_n)$ for partitions $\lambda$ with $\ell(\lambda)\leq k$, and likewise define $\Lambda_{\infty,k} = \ZZ\spanning \{s_\lambda : \ell(\lambda) \leq k\}$. It is well-known that the restriction map $r_n$  defines a bijection $\Lambda_{\infty,k} \to \Lambda_{n,k}$ whenever $k\leq n$ (see, e.g., \cite[Proposition 1.2.1]{Manivel}) and that $fg \in \Lambda_{\infty,j+k}$ whenever $f \in \Lambda_{\infty,j}$ and $g \in \Lambda_{\infty,k}$ (see \cite[\S1.5.4]{Manivel}).
Hence, if we have $(f,g,h) \in \Lambda_{\infty,j}\times  \Lambda_{\infty,k}\times \Lambda_{\infty,j+k}$ and $j+k\leq n$,
then 
\be
\label{equiv}
 h(x_1,\dots,x_n) = f(x_1,\dots,x_n)g(x_1,\dots,x_n) \in \Lambda_n
 \qquad
 \Rightarrow
 \qquad
 h=fg \in \Lambda
\ee
since we may obtain the right identity by applying the inverse of the bijection $r_n : \Lambda_{\infty,j+k} \to \Lambda_{n,j+k}$ to both sides of the equation on the left.

It follows from \cite[Theorem 4.1]{Stan} that $F_u \in \Lambda_{\infty,k}$ if $u \in S_\infty$ has largest descent at most $k$.
In view of Proposition \ref{atomdes-prop}, we thus have $\hat F_u \in \Lambda_{\infty,k}$ (respectively, $\Ffpf_u \in \Lambda_{\infty,k}$) whenever $u$ is an involution (respectively, fixed-point-free involution) with largest descent at most $k$. 
Since $2k \leq n$ and 
$y \in S_n$, and since $\cG_k$ and $r(y)$ both have largest descent at most $k$, we may apply \eqref{equiv} to deduce the desired identities from the formulas 
 in the first paragraph.
\end{proof}

The subalgebra $\LambdaOdd \subset\Lambda$ has a distinguished basis   $\{P_\lambda\}$ indexed by strict partitions, called the \emph{Schur $P$-functions}; see \cite{AS,Stem_proj}. An element $f \in \LambdaOdd$ is \emph{Schur $P$-positive} if it is a nonnegative   linear combination of Schur $P$-functions. 

\begin{theorem}\label{last-cor}
Let $y \in \I(S_\infty)$ be a weakly dominant involution with $k = \kappa(y)$.
Suppose  $D(r(y))$ is
equivalent to a skew shape of the form $\delta_m/\lambda$ for some  $m \in \PP$ and   partition $\lambda\subset \delta_m$, where $\delta_m = (m-1,m-2,\dots,2,1)$.
Then
\[
\hat F_y =  
 s_{\delta_{k+1}} s_{\delta_m/\lambda}
\qquand
\Ffpf_y 
=s_{\delta_k} s_{\delta_m/\lambda},
\]
where the second identity applies only in the case when $y$ is fixed-point-free. Moreover, in this case the symmetric functions $\hat F_y$ and (when defined) $\Ffpf_y$  are Schur $P$-positive.
\end{theorem}

\begin{remark}
We show in \cite{HMP3,HMP4,HMP5} that  $\hat F_y$ and $\Ffpf_z$ are     Schur $P$-positive for all $y \in \I(S_\infty)$ and $z \in \Ifpf(S_\infty)$. Only in the special cases just described does this follow from our present methods.
\end{remark}

\begin{proof}
By Proposition \ref{Fskew-prop} we have $F_{r(y)} = s_{\delta_m/\mu}$, and it is well-known that this skew Schur function belongs to $ \LambdaOdd$: see Proof 2 of \cite[Corollary 7.32]{RSW},  or just adapt the argument in \cite[Proposition 7.17.7]{EC2}.
From this, the formulas for $\hat F_y$ and $\Ffpf_y$ are  immediate by Proposition \ref{Fgrass-prop} and Theorem \ref{last-thm}. For the last assertion, we note that skew Schur functions of the form $s_{\delta_m/\mu}$ are Schur $P$-positive (see \cite{AS}), and that Schur $P$-positivity is closed under products (see \cite[\S8]{Stem_proj}).
\end{proof}

The most important special case of the preceding result is Theorem \ref{F-thm} from the introduction, whose proof we  now give.

\begin{proof}[Proof of Theorem \ref{F-thm}]
One checks that $\kappa(w_n) = \lfloor n/2\rfloor$ and $r(w_n) = w_{\lceil n/2\rceil}$, and that $D(w_n)$ is equivalent to the shape of  $\delta_n$.
Hence, by Theorem \ref{last-cor},  $\hat F_{w_n} = s_{\delta_{k+1}} s_{\delta_m}$ for $k=\lfloor \frac{n}{2}\rfloor$ and $m= \lceil \frac{n}{2}\rceil$
and $\Ffpf_{w_{2n}} = (s_{\delta_n})^2$.
The theorem follows as $\{k+1,m\} = \{p,q\}$ for $p = \lceil\frac{n+1}{2}\rceil$ and $q = \lfloor \frac{n+1}{2}\rfloor$.
\end{proof}

\appendix
\newpage
\section{Index of notation}
\label{not-sect}

The tables below list our common notations, with references to definitions where relevant.

\def\skip{\\[-4pt]}

\begin{center}
\begin{tabular}{l | l r}
\text{Symbol} & \text{Meaning} &\text{Reference}
 \\
 \hline
 $(W,S)$ & An arbitrary Coxeter system \\
  $\I=\I(W)$ & The set of involutions $\{w \in W : w=w^{-1}\}$ & \\ 
  $y\rtimes s$ & For $(y,s) \in \I\times S$, either $ys$ (if $ys=sy$) or $sys$ (if $ys\neq sy$) & \eqref{rtimes-eq} \\
 $\ellhat$ & The involution length function $\I \to \NN$ & \eqref{ellhat-def} \\
  $<_T$ & The two-sided weak order on $\I$ & \eqref{t-eq} \\
\skip
 $S_\infty$ & The group of bijections $\PP \to \PP$ with finite support & \\ 
  $S_n$ & The group of bijections $[n]\to [n]$ viewed as a subgroup of $S_\infty$ & \\ 
$\Ifpf(S_n)$ & The set of fixed-point-free involutions in $S_n$ & \\
$\Ifpf(S_\infty)$  & The union $ \bigcup_{n\in \PP} \Ifpf(S_{2n})$ & \\
$\tilde\I_\fpf$ & The $S_\infty$-conjugacy class of $(1,2)(3,4)(5,6)\cdots$  & \S\ref{invschubert-sect}\\
$\kappa(w)$ & The number of 2-cycles in $w \in \I(S_n)$ & \\
\skip
$\wfpf_n$ & The permutation $(1,2)(3,4)\cdots (2n-1,n) \in \Ifpf(S_{2n})$ & \\
$w_n$ & The longest permutation $[n,n-1,\dots,3,2,1] \in \I(S_n)$  & \\
$\cG_n$ & The Grassmannian involution $(1,n+1)(2,n+2)\cdots(n,2n)$ & \\
$z_\infty$ & An element of $\tilde\I_\fpf$ constructed from $z \in \Ifpf(S_{2n})$ & \eqref{infprod} \\
$r(y)$ & A certain permutation constructed from $y \in \I(S_n)$ & \eqref{r-eq} \\
$1_m \times w$ & The image of $w \in S_k$ in $S_m \times S_k \subset S_{m+k}$ & \eqref{1cross-def} \\
$\cG_{p,q}$ & The  permutation $[q+1,q+2,\dots,n,1,2,\dots,q] \in S_n$ & \\
\skip
$\cR(w)$ & The set of reduced words for $w \in W$ & \S\ref{intro1-sect} \\
$\hat\cR(y,z)$ & The set of involution words from  $y$ to $z$ & \S\ref{intro1-sect} \\
$\hat\cR(y)$ & The set of involution words $\hat\cR(1,y)$  for $y \in \I(W)$ \\
$\cRfpf(z)$ & The set of involution words $\hat\cR(\wfpf_n,z)$ for $z \in \Ifpf(S_{2n})$ & \\
\skip
$\cA(y,z)$ & The set of relative atoms for $y,z \in \I(W)$ & Thm.-Def.~\ref{atoms-thmdef} \\
$\cA(y)$ & The set of atoms $\cA(1,y)$ for $y \in \I(W)$ & \\
$\cAfpf(z)$ & The set of atoms $\cA(\wfpf_n,z)$  for $z \in \Ifpf(S_{2n})$ & \\
$\alpha_{\min}(y)$ & The minimal atom in $\cA(y)$ & \eqref{min-atom-eq} \\
$\beta_{\min}(z)$ & The minimal atom in $\cAfpf(z)$ & \eqref{min-atom-eq}\\
\skip
$D(w)$ & The Rothe diagram of $w \in S_n$ & \eqref{rothe-eq} \\
$\D(y)$ & The involution Rothe diagram of $y \in \I(S_n)$  & \eqref{invol-rothe-eq} \\
$\Dfpf(z)$ & The involution Rothe diagram of $z \in \Ifpf(S_n)$ & \eqref{invol-rothe-eq} \\
$\Ess(D)$ & The essential set of $D\subset \PP\times \PP$ & \eqref{ess-eq} \\
$E_k(u)$ & A certain modified Rothe diagram & \eqref{e-eq} \\
\skip
$c(w)$ & The code of $w \in S_n$  & \eqref{code-def} \\
$\hat c(y)$ & The involution code of $y \in \I(S_n)$ & \eqref{inv-code-eq}  \\
$\hat c_\fpf(z)$ & The involution code of $z \in \Ifpf(S_n)$ & \eqref{inv-code-eq} \\
$\lambda(w)$ & The partition given by sorting $c(w)$ \\
\end{tabular}
\end{center}

\begin{center}
\begin{tabular}{l | l r}
\text{Symbol} & \text{Meaning} &\text{Reference}
 \\
 \hline
$\fkS_w$ & The Schubert polynomial of $w \in S_n$ & \eqref{schub1-eq} \\
$\hat\fkS_{y,z}$ & The involution Schubert polynomial of $y,z \in \I(S_n)$ & \eqref{intro-inv-schub-eq} \\
$\hat\fkS_y$ & The involution Schubert polynomial $\hat\fkS_{1,y}$ of $y \in \I(S_n)$ \\
$\Sfpf_z$ & The involution Schubert polynomial $\hat\fkS_{\wfpf_n,z}$ of $z \in \Ifpf(S_{2n})$ \\
$\fkS_w(x;y)$ & The double Schubert polynomial of $w \in S_n$ & Def.~\ref{doubleschub-def}\\
\skip
$F_w$ & The Stanley symmetric function of $w \in S_n$ & \eqref{F1-eq} \\
$\hat F_{y,z}$ & The involution Stanley symmetric function of $y,z \in \I(S_n)$ & \eqref{intro-inv-schub-eq} \\
$\hat F_y$ & The involution Stanley symmetric function $\hat F_{1,y}$ of $y \in \I(S_n)$ \\
$\Ffpf_z$ & The involution Stanley symmetric function $\hat F_{\wfpf_n,z}$ of $z \in \Ifpf(S_{2n})$ \\
\skip
$\GL_n(\CC)$ & The group of $n\times n$ invertible matrices over $\CC$ \\
$\O_n(\CC)$ & The subgroup of orthogonal matrices in $\GL_n(\CC)$ \\
$\Sp_n(\CC)$ & The subgroup of symplectic matrices in $\GL_n(\CC)$ \\
$B$ & The subgroup of lower triangular matrices in $\GL_n(\CC)$ \\
$B^+$ & The subgroup of upper triangular matrices in $\GL_n(\CC)$ \\
$K$ & Usually $\O_n(\CC)$ or $\Sp_n(\CC)$ \\
\skip
$\Fl(n)$ & The type $A$ flag variety, identified with $B\backslash \GL_n(\CC)$ & \S\ref{cohomology-sect}\\
$X_w$ & The Schubert variety indexed by $w$ & \eqref{mathring-eq} \\
$Y_y^K$ & The closed $K$-orbit in $\Fl(n)$ indexed by $y$ & \eqref{Y-eq} \\
$\Upsilon_y^K$ & Wyser and Yong's cohomology representative for $Y_y^K$ & \S\ref{revisit-sect} \\
$\r_w(i,j)$ & The number of positive integers $t \leq i $ with $w(t) \leq j$ for $w \in S_n$  & \eqref{r-def} \\
\skip
 $\cP_n$ & The polynomial ring $\ZZ[x_1,x_2,\dots,x_n]$ & \\
$\cP_\infty$ & The polynomial ring with infinitely many variables $\ZZ[x_1,x_2,\dots]$  & \\
$\cP_\infty(x;y)$ & The polynomial ring $\ZZ[x_1,y_1,x_2,y_2,x_3,y_3,\dots]$ & \\
$\Lambda$ & The Hopf algebra of symmetric functions over $\ZZ$ & \\
$\Lambda_n$ & The subring of symmetric polynomials in $\cP_n$ & \\
$(\Lambda_n^+)$ &The ideal in $\cP_n$ generated by the non-constant elements of $\Lambda_n$ & \eqref{borel-eq} \\
\skip
$\delta_n$ & The partition $(n-1,n-2,\dots,3,2,1)$ \\
$s_\lambda$ & The Schur function in $\Lambda$ indexed by a partition $\lambda$ \\
$s_{\lambda/\mu}$ & The skew Schur function in $\Lambda$ for partitions $\mu \subset \lambda$  & \\
$f^\lambda$ &The number of standard tableaux of shape $\lambda$ & \\
$p_n$ & The power sum symmetric function $x_1^n + x_2^n + x_3^n+ \dots \in \Lambda$ & \\
$\Gamma$ & The Hopf subalgebra $\QQ\langle p_1,p_3,p_5,\dots\rangle \cap \Lambda $ of $ \Lambda$ & \\
\skip
$\partial_i$ & The $i$th divided difference operator & \eqref{partial-i-eq} \\
$\pi_i$ & The $i$th isobaric divided difference operator &\eqref{pi-i-eq} \\
$\tau_n$ & The operator $\pi_1\pi_2\cdots\pi_n$ & \\
$r_n$ & The operator on power series which sets $x_{n+1}=x_{n+2}=\dots=0$  & \\
$\stb_n$ & The stabilization operator in degree $n$ & \eqref{stab-obs} \\
$\stb_{p,q}$ & A modified stabilization operator & \eqref{stb-p-q-eq} \\
$\stb_{n/p,q}$ & A modified stabilization operator & \eqref{stb-p-q-eq} \\
$\Phi_{p,q}$ & A certain ring homomorphism $\cP_\infty(x;y) \to \cP_{p+q}$ & \eqref{Phi-eq}
\end{tabular}
\end{center}

\end{document}